\documentclass[a4paper, 11pt, twoside]{amsart}
\usepackage[latin1]{inputenc} 
\usepackage[T1]{fontenc}      
\usepackage[english]{babel}  
\usepackage[a4paper,margin=1.25in]{geometry}
\usepackage{xcolor}
\usepackage{hyperref}
\usepackage{enumerate}
\usepackage{amsmath}                            
\usepackage{amsfonts}
\usepackage{amsthm} 
\usepackage{amscd}
\usepackage{amssymb}
\usepackage{verbatim}
\usepackage[all]{xy}
\newtheorem{theo}{Theorem}[section]    
\newtheorem{defi}[theo]{Definition}   
\newtheorem{lem}[theo]{Lemma}
\newtheorem{prop}[theo]{Proposition}
\newtheorem{rem}[theo]{Remark}   

\newtheorem{theoremintro}{Theorem}
\newtheorem{propintro}{Proposition}
\newtheorem{corollaryintro}{Corollary}
\newtheorem{remarkintro}{Remark}
\newtheorem*{theorem*}{Theorem}
   
\newcommand{\R}{\mathbb{R}}  
 
\newcommand{\N}{\mathbb{N}} 
\newcommand{\Z}{\mathbb{Z}} 
\newcommand{\C}{\mathbb{C}} 
 
\newcommand{\D}{\mathbb{D}}

\newcommand{\rs}{\mathbb{P}^1}

\newcommand{\limn}{\lim_{n \rightarrow \infty}}
\newcommand{\id}{\mathrm{Id}}

\newcommand{\mcal}{\mathcal{M}}
\newcommand{\ocal}{\mathcal{O}}

\newcommand{\bcal}{\mathcal{B}}

\newcommand{\jac}{\mathrm{Jac}}

\newcommand{\eps}{\varepsilon}
\newcommand{\lcal}{\mathcal{L}}

\newcommand{\re}{\mathrm{Re}\,}
\newcommand{\im}{\mathrm{Im}}
\renewcommand{\H}{\mathbb{H}}
\newcommand{\ptwo}{\mathbb{P}^2}
\newcommand{\rcal}{\mathcal{R}}

\title{Parabolic implosion in dimension 2}                         
\author{Matthieu Astorg}
\address{Matthieu Astorg\\ Institut Denis Poisson, Universit\'e d'Orl\'eans}
\email{matthieu.astorg@univ-orleans.fr}
\author{Lorena L\'opez-Hernanz}
\address{Lorena L\'opez-Hernanz\\ Departamento de F\'isica y Matem\'aticas, Universidad de Alcal\'a}
\email{lorena.lopezh@uah.es}
\author{Jasmin Raissy}
\address{Jasmin Raissy, Univ. Bordeaux, CNRS, Bordeaux INP, IMB, UMR 5251, F-33400 Talence, France \& Institut Universitaire de France (IUF)}
\email{jasmin.raissy@math.u-bordeaux.fr}

\begin{document}

\begin{abstract}
In this paper, we extend the theory of parabolic implosion in complex dimension 2 to the case of holomorphic maps tangent to the identity at order 2. We investigate the bifurcation phenomena that occur when a fully parabolic fixed point is perturbed. Under the assumption of a non-degenerate characteristic direction with a formal invariant curve and director $\alpha$ satisfying $\re\alpha> 2$, we establish the existence of Lavaurs maps as limits of iterates $f_{\epsilon_n}^n$ for specific sequences of the perturbation parameter $\epsilon_n$. Finally, we apply these results to prove the discontinuity of the Julia sets $J_1$ and $J_2$ for holomorphic endomorphisms of $\mathbb{P}^2$, generalizing classical one-dimensional results to this higher-dimensional setting.
\end{abstract}

\maketitle

\section{Introduction}

Parabolic implosion is the study of the bifurcation phenomena which occur when 
a multiple (i.e., parabolic) fixed point is perturbed and splits into several fixed points or periodic cycles.
It was first developed by Lavaurs in his PhD thesis (\cite{lavaurs1989systemes}).
A first consequence of this theory is a precise description of the discontinuity (\emph{enrichment}) of Julia sets with respect to the parameter, in the presence of a non-persistent parabolic cycle. 
This was used by Shishikura (\cite{shishikura1998hausdorff}) to prove that the boundary of the Mandelbrot set has Hausdorff dimension 2.
A refinement of parabolic implosion (\emph{near-parabolic renormalization}), developed by Inou and Shishikura (\cite{inou2006renormalization}), has
led to remarkable results, such as the construction of quadratic Julia sets with positive area by Buff and Ch\'eritat (\cite{buff2012quadratic})
or progress towards the hyperbolicity conjecture (\cite{cheraghi2015satellite}).

More recently, the theory of parabolic implosion has started to develop and to find successful applications in higher dimension.
In \cite{BSU17semi}, Bedford, Smillie and Ueda develop a parabolic implosion theory in the setting of semi-parabolic diffeomorphisms in dimension 2, i.e., in the case of a fixed point
with one attracting direction and the other one with multiplier equal to 1.
 In particular, in the important case of a dissipative H\'enon map, they were able to deduce the discontinuity of 
several dynamically defined sets (including the forward Julia set $J^+$ and the closure $J^*$ of the saddle periodic points) with respect to the parameter.
Building on their result, Bianchi and the first named author proved in \cite{AB3} the existence of perturbations of such H\'enon maps 
whose forward Julia set $J^+$ has large Hausdorff dimension.

In \cite{DL15stability},  Dujardin and Lyubich adapted the results of \cite{BSU17semi} to construct homoclinic tangencies for perturbations of dissipative H\'enon maps with a semi-parabolic periodic cycle, with applications 
to bifurcation theory.
In \cite{ABDPR16}, the authors used parabolic implosion techniques to construct the first examples of polynomial maps (in dimension 2) with a wandering Fatou component (see also \cite{astorg2023wandering}, \cite{astorg2024dynamics}). The dynamical systems under consideration are polynomial skew-products, hence the techniques employed
can be seen as a non-autonomous version of one-dimensional parabolic implosion.
Finally, in \cite{bianchi19parabolic}, Bianchi obtained results analogous to those of \cite{BSU17semi} but for the more difficult case of maps with a fully parabolic fixed point, i.e., in the case where the differential at the fixed point is the identity. The purpose of this article is to extend the results from \cite{bianchi19parabolic}.

\medskip

Let us now provide a quick overview of classical parabolic implosion in dimension one, in the simplest case of a parabolic fixed point with just one attracting and one repelling petal.
Let $f: \C \to \C$ be a holomorphic map of the form $f(z)=z+z^2+\ocal(z^3)$. 
By the classical Leau-Fatou theorem, there exists $r>0$ and univalent maps $\phi^\iota : \D(-r,r) \to \C$ and $\phi^o: \D(r,r) \to \C$ such that 
\begin{enumerate}
\item $P^\iota:=\D(-r,r)$ is forward invariant under $f$, and $P^o:=\D(r,r)$ is invariant under the branch of $f^{-1}$ fixing the origin;
\item $\phi^\iota \circ f = \phi^\iota + 1$, and $\phi^o \circ f = \phi^o +1$.
\end{enumerate}
The domains $P^\iota$ and $P^o$ are respectively called incoming and outgoing petals and the maps $\phi^\iota$ and $\phi^o$ are called incoming and outgoing Fatou coordinates.
Let 
$$\bcal:=\{z \in \C: f^n(z) \to 0 \text{ and } f^n(z) \neq 0 \quad \forall n \geq 0  \}$$ be the \emph{parabolic basin}. 
Then $P^\iota \subset \bcal$, and moreover $\bcal = \bigcup_{n \geq 0} f^{-n}(P^\iota)$. The incoming Fatou coordinate extends to a holomorphic map $\phi^\iota: \bcal \to \C$ and the inverse $(\phi^o)^{-1}$ of the outgoing Fatou coordinate extends to a holomorphic map $\psi^o : \C \to \C$, called the outgoing Fatou parametrization. We refer the reader to \cite{milnor2011dynamics}  for details.
In particular, for any $\sigma \in \C$, the change of coordinate $\lcal_\sigma:=\psi^o \circ T_\sigma \circ \phi^\iota$ is well-defined on $\bcal$, 
where $T_\sigma(z):=z+\sigma$ is the translation of vector $\sigma$. It is called the \emph{Lavaurs map} of phase $\sigma$.\footnote{Fatou coordinates are in fact only unique up to addition by a constant; thus $\lcal_\sigma$ may be thought of as $(\phi^o)^{-1} \circ \phi^\iota$, for a different choice of $\phi^\iota$. In practice, we will work with some unique normalizations of Fatou coordinates based on their formal expansion at $0$.}

Consider now the family of perturbations $f_\eps(z) = f(z) + \eps^2$, $\eps\in \C$.
For $\eps$ small but non-zero, the double fixed point at the origin for $f$ splits into 2 simple fixed points for $f_\eps$, of the form $z^\pm(\eps) = \pm i \eps + \ocal(\eps^2)$. 
If we take {say} $\eps>0$, then Lavaurs proved that orbits under $f_\eps$ starting from a point in $P^\iota$ will approach the origin, then cross the "gate" given by the vertical segment $[z^-(\eps), z^+(\eps)]$ between $z^-(\eps)$ and $z^+(\eps)$ and then move away from the origin inside $P^o$. Since $f_\eps$ is close to the identity near $0$, as $\eps \to 0$ it takes more and more iterations to do this. At the limit, we obtain in this way a "transit map" from $P^\iota$ to $P^o$, which is useful for studying the dynamics of $f_\eps$.
It turns out that this transit map is exactly the Lavaurs map defined above.
More precisely:

\begin{theorem*}[Lavaurs, \cite{lavaurs1989systemes}]
Let $(\eps_n)_{n \in \N}$ be a sequence of complex numbers, and $\sigma \in \C$. Assume that $\limn \left(n-\pi/\eps_n\right)= \sigma$. Then $f_{\eps_n}^n \to \lcal_\sigma$ locally uniformly on $\bcal$.
\end{theorem*}

\medskip

We now move on to the setting of complex dimension 2. 
Let $f: U{\subset\C^2} \to \C^2$  be a holomorphic map on a neighborhood $U$ of the origin, with a power series expansion of the form
$$f = \mathrm{Id} + P_2 + P_{3} + \ldots$$ 
where the $P_j$ are homogeneous degree $j$ polynomial maps from $\C^2$ to $\C^2$, and $P_2\not\equiv 0$. Such a map is called \emph{tangent to the identity at order 2}.
Following Hakim~\cite{hakim1998analytic} and \'Ecalle~\cite{ecalleresurgentes}, we say that $v \in \mathbb{C}^2 \setminus \{(0,0)\}$ is a \emph{characteristic direction} for $f$ if there exists  $\lambda \in \mathbb{C}$ so that
$P_k(v) = \lambda v$. 
If $\lambda \neq 0$ then $v$ is said to be \emph{non-degenerate}.  We shall denote by $v\mapsto [v]$ the canonical projection of
$\mathbb{C}^{2}\backslash \{(0,0)\}$ onto $\mathbb{P}^{1}$.  The \emph{director} of a non-degenerate characteristic direction $v$ is the eigenvalue of the linear operator 
$$
d(P_2)_{[v]}-\id:T_{[v]}\mathbb{P}^1\rightarrow T_{[v]}\mathbb{P}^1.
$$
If the real part of the director of a non-degenerate characteristic direction $v$ is strictly positive, Hakim proved in \cite{hakim1997transformations} that for any $C>0$ there exist incoming and outgoing petals $P_C^\iota$ and $P_C^o$ and incoming/outgoing Fatou coordinates $\Phi^{\iota/o}: P_C^{\iota/o} \to \C^2$ conjugating $f$ to a translation of vector $(1,0)$
(see Propositions \ref{prop:incomfatou} and \ref{prop:outgofatou} for details).
Let
$$\bcal_{U,v}:=\{(x,y)\in U: f^n(x,y)\to(0,0) \text{ tangentially to } v\}$$
denote the \emph{parabolic basin associated to $v$}.
Similarly to  the one-dimensional case, we have that 
$\bcal_{U,v} = \bigcup_{C>0}\bigcup_{n \geq 0} f^{-n}(P_C^\iota)$ (see \cite{lopez2025flower}).
 
We say that a formal non-singular curve $\mathcal{C}$ is invariant for $f$ if given a parametrization $\gamma(t)$ of $\mathcal{C}$ (i.e., $\gamma(t) \in \C[[t]]^2$ with $\gamma(0)=(0,0)$ and  $\gamma'(0) \neq (0,0)$) there exists $h \in t\C[[t]]$ with $h'(0)\ne0$ such that 
$$f \circ \gamma = \gamma \circ h.$$
The tangent of $\mathcal{C}$ is, by definition, $\C \cdot \gamma'(0)$.

From now on, we will assume that $f: U{\subset\C^2} \to \C^2$ is a holomorphic map defined on a neighborhood $U$ of the origin and which satisfies the following assumption: 

\begin{enumerate}
\item[$(H_1)$] The map $f$ is tangent to the identity at order $2$ and has a non-degenerate characteristic direction $v$, with a formal {non-singular} invariant curve $\mathcal{C}$ tangent to $v$, and a director $\alpha$ such that $\re\alpha>2$.
\end{enumerate}

It is worth mentioning that the existence of a non-singular formal invariant curve tangent to a non-degenerate characteristic direction $v$ is a generic hypothesis. It is equivalent to the existence of an analytic curve $C$ tangent to $v$ which is preserved up to order $k=\lfloor \re\alpha\rfloor+2$ in the following sense: if $\gamma(t)$ is a parametrization of $C$ then there exists $h \in t\C\{t\}$ such that
$$f \circ \gamma - \gamma \circ h  \in \langle t^{k+1} \rangle$$
and this condition is always satisfied when $\alpha\not\in\N$ (see [\cite{hakim1998analytic}, Section 3]). 

We now state a first, non-technical version of our main result:

\begin{theoremintro}[Non-technical version]\label{th:v1}
Let $f:U\subset\C^2\to\C^2$ be a holomorphic map satisfying $(H_1)$.  Then for {any} $q \in \C$ there exists a holomorphic family of holomorphic maps $(f_\eps: U \to \C^2)_{\eps \in \D}$ with $f_0=f$ such that for any $\sigma \in \C$ and for any compact set $K \subset \bcal_{U,v}$ there exist $N \in \N$ and a sequence 
$\eps_n \to 0$ such that 
$f_{\eps_n}^{n-N} \to (\Phi^o)^{-1} \circ A_{\sigma-N,q} \circ \Phi^\iota$
uniformly on $K$, where 
$A_{\sigma-N,q}(X,Y):=(X+\sigma-N, e^{\pi q} Y)$. 
\end{theoremintro}

Since the maps $f$ and $f_\eps$ are only defined on $U$, in general the iterates $f_{\eps_n}^n$ may not be well-defined; however Theorem~\ref{th:v1} implies that there is a constant integer $N$ such that $f_{\eps_n}^{n-N}$ is well-defined on $K$ for all $n$ large enough. Moreover, since the map
$$\lcal_{\sigma-N, q}=(\Phi^o)^{-1} \circ A_{\sigma-N,q} \circ \Phi^\iota$$ satisfies 
$$f \circ \lcal_{\sigma-N, q} = \lcal_{\sigma-N, q} \circ f = \lcal_{\sigma-N+1, q},$$
it is not difficult to see that we also have 
$f_{\eps_n}^{n-N'} \to \lcal_{\sigma-N',q}$ 
for any $N' \in \Z$ such that  $\lcal_{\sigma-N',q}(K) \subset U$. In particular, in the case where the maps $f_\eps$ are global endomorphisms of a complex manifold we may simply take $N=0$.

We can be more precise on the requirements for the family of perturbations $(f_\eps)_{\eps \in \D}$, and interpret  the constants $\sigma$ and $q$ in terms of $f_\eps$. We will make the following assumptions:

\begin{enumerate}
\item[$(H_2)$] The family $(f_\eps)_{\eps \in \D}$ preserves the formal curve $\mathcal{C}$ to order $m+1$, where $m:=\lfloor\re\alpha\rfloor+1$, in the following sense: {if $\gamma(t)$ is a parametrization of $\mathcal{C}$,} there exists $h_\eps \in \ocal(\D)[[t]]$ such that
$$f_\eps \circ \gamma - \gamma \circ h_\eps \in \langle t^{m+2}, \eps t^{m+1}, \ldots, \eps^{m+2} \rangle.$$
\item[$(H_3)$] For all $\eps$ in a neighborhood of $0$, $f_\eps$ has exactly 4 fixed points $z_i(\eps)$ near the origin ($1 \leq i \leq 4$) and counted with multiplicity, which depend holomorphically on $\eps$ and such that $z_i'(0):=\frac{d}{d\eps}_{|\eps=0} z_i(\eps)$ are non-zero and pairwise distinct.
\end{enumerate}

Note that the existence of 4 fixed points depending holomorphically on $\eps$ is always satisfied up to passing to a branched cover in parameter space (although doing so affects the derivatives $z_i'(0)$). 

{With a slight abuse of notation we will say that $(f_\eps)_{\eps \in \D}$  \emph{satisfies $(H_1)-(H_3)$} when $f_0$ satisfies $(H_1)$ and $(f_\eps)_{\eps \in \D^*}$ satisfies $(H_2)$ and $(H_3)$.}

We can interpret the constants $\sigma$ and $q$ in terms of the multipliers of the fixed points of $f_{\eps}$. To make this precise, we first need the following Proposition, whose proof is deferred until the next section:

\begin{propintro}\label{prop:2in1intro}
Assume that $(f_\eps: U \to \C^2)_{\eps \in \D}$ satisfies $(H_1)-(H_3)$. Then there are exactly two fixed points of $f_\eps$, say $z_1(\eps),z_2(\eps)$, which are asymptotically tangent to $v$ for $\eps$ small, i.e.,
$$z_i'(0) \in \C^* v,  \quad 1 \leq i \leq 2.$$
Moreover, if the eigenvalues $\lambda_i(\eps), \mu_i(\eps)$ of $z_i(\eps)$ satisfy that $\lambda_i'(0)\neq \mu_i'(0)$ then one of the eigenspaces of $z_i(\eps)$ tends to $\C v$ as $\eps\to0$, and the condition $\lambda_i'(0)\neq \mu_i'(0)$ holds for at least one of the fixed points.
\end{propintro}

\begin{remarkintro}\label{rem:H_3'}
If $(H_1)$ and $(H_2)$ are satisfied, then the proof of Proposition \ref{prop:2in1intro} will show that 
$(H_3)$ can be replaced by the following slightly weaker condition: 
\begin{enumerate}
\item[$(H_3')$] The jet of order $m+1$ of $h_{\eps}$ has two fixed points $w_\pm(\eps)$, and 
$w_+'(0) \neq w_-'(0)$.
\end{enumerate}
Note that even though the fixed points $w_\pm(\eps)$ depend on the higher order terms of $h_\eps$, their derivatives at $0$ do not.
\end{remarkintro}

We can now state a second version of our results, which is more explicit on the requirements on the family of perturbations $(f_\eps)_{\eps \in \D}$. To do so, we will use the following convention: let $z_1(\eps)$ and $z_2(\eps)$ be the fixed points from Proposition~\ref{prop:2in1intro} and let $\lambda_1(\eps),\mu_1(\eps)$ and $\lambda_2(\eps),\mu_2(\eps)$ be their eigenvalues. If $\lambda_i'(0)\neq\mu_i'(0)$, 
we denote by $\rho_T^i(\eps)$ the eigenvalue whose eigenspace tends to $\C v$ as $\eps\to 0$ and by $\rho_N^i(\eps)$ the other one; if 
$\lambda_i'(0)=\mu_i'(0)$, we assign the names $\rho_T^i(\eps)$ and $\rho_N^i(\eps)$ indifferently to the two eigenvalues $\lambda_i(\eps)$ and $\mu_i(\eps)$.

\begin{theoremintro}[Coordinate-free version]\label{th:v2}
Let $(f_\eps: U \to \C^2)_{\eps \in \D}$ be a family of holomorphic maps satisfying $(H_1)-(H_3)$. Let $z_1(\eps), z_2(\eps)$ be the fixed points from Proposition~\ref{prop:2in1intro} and denote by $\rho_T^i(\eps)$ and $\rho_N^i(\eps)$ ($1\le i\le 2$) their eigenvalues, using the convention above. Let  
$$q:=\lim_{\eps\to0} \frac1{\eps}\frac{\rho_N^1(\eps)+\rho_N^2(\eps)-2}{\rho_T^1(\eps)+\rho_T^2(\eps)}.$$ 
There exists a constant $\sigma_0 \in \C$ such that for any $1\le i\le2$, for any $\sigma \in \C$, for any sequence $(\eps_n)_{n \in \N}$  such that 
\begin{equation*}
\frac{2i\pi}{\rho_T^i(\eps_n)-1} = n- \sigma + o(1),
\end{equation*}
and for any compact set $K \subset \bcal_{U,v}$, there exists $N \in \N$ such that
$$f_{\eps_n}^{n-N} \to (\Phi^o)^{-1} \circ A_{\sigma+\sigma_0-N,q} \circ \Phi^\iota$$ 
uniformly on $K$, where
$A_{\sigma+\sigma_0-N,q}(X,Y):=(X+\sigma+\sigma_0-N, e^{\pi q} Y)$.
\end{theoremintro}

We finally state a third version of our main result, explicitly expressed in coordinates:

\begin{theoremintro}[Coordinate version]\label{th:explicit}
Let 
$$g_\eps(x,y) = (x + (x^2+\eps^2)a_\eps(x)+ y b_\eps(x,y),  y + y c_\eps(x,y) + d_\eps(x) )$$
be a family of holomorphic maps defined on a neighborhood $U$ of $(0,0)$, where $a_\eps, b_\eps, c_\eps$ and $d_\eps$ depend holomorphically on $\eps$, 
and assume that
\begin{enumerate}
\item $a_{0}(0)=1$, $b_0(0,0)=0$
\item $c_{\eps}(x,y)= \eta x + q  \eps +c y + \ocal_2(x,y,\eps)$, with $\re\eta>3$, $q, {c}\in\C$ and $d_{\eps}(x) =\ocal(x^{m+3})+ \eps\ocal_{m+1}(x,\eps)$ where $m =\lfloor\re \eta\rfloor$. 
\end{enumerate}
Let $(\eps_n)$ be a sequence such that $n-\pi/\eps_n = \sigma + o(1)$.
Then for any compact set $K \subset \bcal_{U,(1,0)}$ there exists $N \in \N$ such that 
$$g_{\eps_n}^{n-N} \to (\Phi^o)^{-1} \circ A_{\sigma-N,q} \circ \Phi^\iota$$
uniformly on $K$, where $A_{\sigma-N,q}(X,Y):=(X+\sigma-N, e^{\pi q} Y)$.
\end{theoremintro}

Theorem \ref{th:explicit} can be interpreted as a generalization of Bianchi's main result in \cite{bianchi19parabolic}, Theorem 1.4. Let us comment here on the differences between Theorem~\ref{th:explicit} and \cite[Theorem 1.4]{bianchi19parabolic}. First, \cite[Theorem 1.4]{bianchi19parabolic} applies only to maps satisfying a strong assumption, namely that they leave invariant the 3 complex lines $x=\pm {i\eps}$ and $y=0$ (which amounts to taking $b_\eps(x,y)= d_\eps(x) = 0$ with our notations). If that is the case, then  the formal curve corresponding to our assumptions $(H_1)$ and $(H_2)$ is the curve $y=0$, and it is invariant for the whole family of perturbations. Secondly, \cite[Theorem 1.4]{bianchi19parabolic} only proves convergence near the line $y=0$ (instead of the whole parabolic basin associated to $v=(1,0)$), and only up to extraction. In particular, it does not rule out the possibility of the sequence $(g_{\eps_n}^n)_{n \in \N}$ having more than one limit value. Finally, in  \cite[Theorem 1.4]{bianchi19parabolic} the possible limits of $(g_{\eps_n}^n)_{n \in \N}$ are not described explicitly as maps of the form $(\Phi^o)^{-1} \circ A_{\sigma,q} \circ \Phi^\iota$, and only depend on the parameter $\sigma$ since in his case $q=0$.
On the other hand, we must note that \cite[Theorem 1.4]{bianchi19parabolic} only assumes $\re\eta>1$, compared to our assumption that $\re\eta>3$, so our results do not strictly imply his.

\medskip

Let us now give an application case of our main results. A holomorphic endomorphism of $\ptwo$ is a map which may be written in homogeneous coordinates as 
$$f([z_0:z_1:z_2]) = [P_0(z): P_1(z): P_2(z)],$$
where $P_0, P_1, P_2: \C^3 \to \C$ are homogeneous polynomials of degree $d$ with no common factors. The integer $d$ is called the algebraic degree of $f$.
Given such an endomorphism $f$, one can define two distinct notions of Julia sets: the set $J_1(f)$, which may be defined as the non-normality locus; and the set $J_2(f)$, also sometimes called the small Julia set, which may be defined as the support of the unique measure of maximal entropy. Alternatively, $J_m(f)$  ($1 \leq m \leq 2$) may be defined as the support of $T_f^{\wedge m}$, where $T_f$ is the so-called Green current of $f$ (this construction is not specific to the case of dimension 2).  In general, $J_2(f) \subsetneq J_1(f)$, and even for very simple maps 
(such as product maps) there is no equality.
We refer the reader to the survey \cite{dinh2010dynamics} for more details.

In the following, for any $(\sigma,q) \in \C^2$ we will let $\lcal_{\sigma, q}:=\Psi^o \circ A_{\sigma,q} \circ \Phi^\iota$, where $\Psi^o$ is the extension of $(\Phi^o)^{-1}$ to $\C^2$.
For any endomorphism $f$ of $\ptwo$ satisfying $(H_1)$, we will set 
$$J^1(f, \lcal_{\sigma, q}) :=   \overline{\{  z \in \ptwo: \, \exists p \in J_1(f) \ \exists n \in \N \ \, \lcal_{\sigma, q}^n(p) =z \}}$$

\begin{corollaryintro}[{Compare to \cite[Theorem~1.6]{bianchi19parabolic}}]\label{coro:discbigjulia}
Let $f: \ptwo \to \ptwo$ be an endomorphism satisfying $(H_1)$ and of algebraic degree $d > \re\alpha+1$. Then for any $q \in \C$ there exists a family $(f_\eps)_{\eps \in \D}$ of endomorphisms of $\ptwo$ of algebraic degree $d$ satisfying $(H_1)-(H_3)$, with $q$ as Theorem~\ref{th:v2}. Moreover,  for any $\sigma \in \C$,
$$\liminf_{n \to +\infty} J_1(f_{\eps_n}) \supset J_1(f, \lcal_{\sigma,q})$$
where $(\eps_n)_{n \in \N}$ is as in Theorem~\ref{th:v2}. 
\end{corollaryintro}

 \begin{corollaryintro}\label{coro:endo}
Let $f: \ptwo \to \ptwo$ be an endomorphism satisfying $(H_1)$ and of algebraic degree $d > \re\alpha+1$. Assume moreover that $\Psi^o(\C^2) \cap J_2(f) \neq \emptyset$.
Then the map $\mathcal{H}_d(\ptwo) \ni g \mapsto J_2(g)$ is discontinuous at $f$, where  $\mathcal{H}_d(\ptwo)$ 
denotes the space of degree $d$ endomorphisms of $\ptwo$.
\end{corollaryintro}

Even if we drop the assumption that $\Psi^o(\C^2) \cap J_2(f) \neq \emptyset$, our arguments still prove the discontinuity of the closure of the set of repelling periodic points; however, as mentioned above, the Julia set $J_2$ may be smaller than this closure.
This hypothesis  is not easy to check in practice;
let us however give a concrete example. Let $\eta \in \C$ with $\re\eta>3$, and let $d> \re\eta$.
Let 
$$f(x,y) = (x+x^2+axy+by^2 + x^d, y + \eta xy+cy^2+y^d).$$
The polynomial map $f:\C^2 \to \C^2$ extends to an endomorphism of $\ptwo$. If $(a,b)=(0,0)$, then the map $f$ is a polynomial skew-product. 
In that case, $\Psi^o$ is of the form $\Psi^o(x,y) = (\Psi_p^o(x), \Psi_2^o(x,y))$, where $\Psi_p^o$ is an outgoing Fatou parametrization of the base polynomial map $p(z):=z+z^2+z^d$. By \cite{jonsson1999dynamics}, $J_2(f):=\overline{\bigcup_{z \in J(p)}  J_z }$, where $J(p)$ denotes the Julia set of $p$ and 
$J_z$ is the non-normality locus of $\{f^n: n \in \N\}$ restricted to the vertical line $x=z$. Since $\Psi_p^o$ is non-constant and entire, it omits at most one value, so there exists $x_0 \in J(p)$ and $X_0 \in \C$ such that $\Psi_p^o(X_0) = x_0$. Similarly, the map $Y \mapsto {\Psi_2^o}(X_0, Y)$ is entire and non-constant and $J_{x_0}$ is uncountable; so $\Psi^o(\C^2) \cap (J_2(f)) \neq \emptyset$. 
Now, the set $J_2(f)$ varies lower semi-continuously with respect to the parameters $(a,b,c)$; and the map $\Psi^o$ depends holomorphically (hence continuously)
on $(a,b,c)$. (For a proof of this fact in dimension one, see the Appendix in \cite{ABDPR16}; the argument remains valid in higher dimension).
Therefore, there exists some open set $W \subset \C^3$ such that for all $(a,b,c) \in \C^3$, the map $f$ satisfies  $\Psi_f^o(\C^2) \cap J_2(f) \neq \emptyset$.

\subsection*{Acknowledgements} 
The first and third author are partially supported by the ANR PADAWAN /ANR-21-CE40-0012-01, ANR DynAtrois / ANR-24-CE40-1163, ANR TI\-GerS / ANR-24-CE40-3604 and the PHC Galileo program, under the project ``From rational to transcendental: complex dynamics and parameter spaces''.  The second author is partially supported by Ministerio de Ciencia e Innovaci\'on, Spain, PID2022-139631NB-I00 and by ANR TIGerS/ANR-24-CE40-3604. The third author is partially supported also by the Institut Universitaire de France (IUF).

\subsection*{Outline of the paper}
In Section~\ref{sec:diffvers}, we give a proof of Theorem~\ref{th:v1} and Theorem~\ref{th:v2} assuming Theorem~\ref{th:explicit}. Sections 3 to 5 are devoted to the proof of Theorem~\ref{th:explicit}. In Section~\ref{sec:fatoucoord}, we introduce the incoming and outgoing petals for $g_0$ and recall the construction of Fatou coordinates and compute their asymptotics. Section~\ref{sec:approximate} is devoted to the construction of so-called \emph{approximate Fatou coordinates} for $g_{\eps}$, which are in a sense close to the actual Fatou coordinates of $g_0$ and which nearly conjugate the dynamics of $g_{\eps_n}$ to a translation. In Section~\ref{sec:control} we provide precise estimates of the orbit under $g_{\eps_n}$ in the parabolic basin and complete the proof of Theorem~\ref{th:explicit}. Finally, Corollaries~\ref{coro:discbigjulia} and \ref{coro:endo} are proved in Section~\ref{sec:coro}.

\section{{Proof of Theorems~\ref{th:v1} and \ref{th:v2} from Theorem~\ref{th:explicit}}}\label{sec:diffvers}

Let us first show how Theorem~\ref{th:explicit} implies Theorem~\ref{th:v1}. Let $f:U\subset\C^2\to\C^2$ be a holomorphic map satisfying $(H_1)$. If we choose coordinates $(x,y)$ such that $v=(1,0)$ we have that
$$f(x,y)=\left(x+\lambda x^2+\ocal(x^3,xy,y^2), y+\eta xy+\ocal(x^2y,y^2,x^3)\right)$$
with $\lambda\neq0$ and $\eta/\lambda=\alpha+1$, where $\alpha$ is the director of $v$. Up to conjugating by the linear map $(x,y)\mapsto (\lambda x,y)$, we can assume that $\lambda=1$, so $\eta=\alpha+1$. In those coordinates, the formal invariant curve $\mathcal C$ has a parametrization $\gamma(t)=(t,\zeta(t))$, with $\zeta(t)\in t^2\C[[t]]$. If we take $m:=\lfloor \re\eta\rfloor$ and $\Psi(x,y):=(x,y-J_{m+2}\zeta(x))$, where $J_{m+2}\zeta$ is the jet of order $m+2$ of $\zeta$, we have that $g_0(x,y)=\Psi\circ f\circ \Psi^{-1}$ has the form
$$g_0(x,y)=\left(x+x^2a_0(x)+yb_0(x,y), y+yc_0(x,y)+\ocal(x^{m+3})\right),$$
with $a_0(x)=1+\ocal(x)$, $b_0(0,0)=0$ and $c_0(x,y)=\eta x+\ocal(x^2,y)$. Then Theorem~\ref{th:v1} follows immediately considering a family $(g_\eps)$ and a sequence $(\eps_n)$ as in Theorem~\ref{th:explicit} and taking $f_{\eps}:=\Psi^{-1}\circ g_{\eps}\circ \Psi$. 

\strut 

Let us now obtain Theorem~\ref{th:v2} from Theorem~\ref{th:explicit}. Consider a family $(f_\eps: U\subset\C^2 \to \C^2)_{\eps \in \D}$ of holomorphic maps satisfying $(H_1)-(H_3)$. We will start making several successive changes of coordinates until we obtain the form of Theorem~\ref{th:explicit}. 

As above, we first choose coordinates in which $v=(1,0)$ and the curve $\mathcal{C}$ has a parametrization $\gamma(t) := (t,\zeta(t))$, with $\zeta(t)\in t^2\C[[t]]$. 
Let $m:=\lfloor \re \alpha\rfloor +1$ and $\Psi(x,y):=(x,y-J_{m+1}\zeta(x))$, where $J_{m+1}\zeta$ is the jet of order $m+1$ of $\zeta$, and set $\widetilde  g_\eps:=\Psi \circ f_\eps \circ \Psi^{-1}$. Since the family $(f_{\eps})$ preserves $\mathcal{C}$ up to order $m+1$ by hypothesis $(H_2)$, we have
$$\widetilde  g_\eps(x,0)= (p_\eps(x),\ocal_{m+2}(x,\eps)),$$
where $p_\eps$ depends holomorphically on $\eps$, and 
$p_0(x)=x + \lambda x^2 + \ocal(x^3)$ for some constant $\lambda \neq 0$ (as above, the fact that $\lambda \neq 0$ follows from the assumptions of $f_0$ having order 2 and $v$ being non-degenerate). Up to conjugating by the linear map $(x,y)\mapsto(\lambda x,y)$, we may assume without loss of generality that $\lambda=1$.

In particular, there exist holomorphic functions $\widetilde  b_\eps$ and $\widetilde  c_\eps$ such that 
\begin{equation}\label{eq:tildeg}
\widetilde  g_\eps(x,y) = (p_\eps(x) + y \widetilde  b_\eps(x,y), y (1 + \widetilde  c_\eps(x,y)) + \ocal_{m+2}(x,\eps)).
\end{equation}
Since $\widetilde  g_0$ is tangent to the identity, we also have $\widetilde  b_0(0,0) = \widetilde  c_0(0,0) = 0$. Moreover, by hypotheses $(H_3)$, $\widetilde  g_{\eps}$ has 4 fixed points $\widetilde  z_i(\eps)=\Psi(z_i(\eps))$ ($1\le i\le 4$) near the origin, which depend holomorphically on $\eps$ and satisfy that $\widetilde  z_i'(0)$ are non-zero and pairwise different.

\begin{lem}\label{lem:fptgtv}
There are exactly two fixed points $\widetilde  z_1(\eps), \widetilde  z_2(\eps)$ of $\widetilde  g_\eps$ (counting multiplicity) which are asymptotically tangent to $v$ for $\eps$ small, i.e.,
$$\widetilde  z_i'(0) \in \C^* v,  \quad 1 \leq i \leq 2.$$
Moreover, if $w_\pm(\eps)$ denote the two fixed points of $p_\eps$, then $w_{\pm}(\eps)$ depend holomorphically on $\eps$ and $\widetilde  z_i(\eps) = (w_\pm(\eps),0) + \ocal(\eps^2)$. In particular, $w_+'(0) \neq w_-'(0)$. 
\end{lem}

\begin{proof}
Let us write $\widetilde  g_\eps(x,y) = \sum_{n \geq 0} P_n(x,y,\eps)$, where $P_n: \C^3 \to \C^2$ is homogeneous polynomial map of degree $n$. Since $\widetilde  g_0$ is tangent to the identity, we have $P_0=0$ and 
$P_1(x,y,\eps) = (x,y) + (\gamma \eps, \delta \eps)$ for some $\gamma, \delta \in \C$.
We claim that under our assumptions, we must have $\gamma= \delta = 0$. Indeed, since $\widetilde  g_\eps(\widetilde  z_i(\eps) ) = \sum_{n \geq 0} P_n(\widetilde  z_i(\eps),\eps)$ and $\widetilde  z_i(0) =0$, by differentiating  $\widetilde  g_\eps(\widetilde  z_i(\eps)$, we obtain 
$$
(\gamma,\delta) = \left.\partial_\eps \widetilde  g_\eps(\widetilde  z_i(\eps)\right|_{\eps=0} = (0,0).
$$
In particular, this means that 
$$p_\eps(x) = x+ x^2 + a_{1,1} \eps x + a_{0,2} \eps^2 + \ocal_3(x,\eps).$$
 By Weierstrass' Preparation Theorem applied to $p_\eps - \id$, there exists holomorphic maps $\eps \mapsto \alpha(\eps), \eps \mapsto  \beta(\eps)$ and $(\eps,x) \mapsto u(\eps,x)$ 
with $\alpha(0)=\beta(0)=0$ and $u(0,0)=1$ such that
$$p_\eps(x) = x + (x^2+ \alpha(\eps) x + \beta(\eps)) u(\eps,x).$$
Moreover, $\alpha(\eps) = a_{1,1} \eps + \ocal(\eps^2)$ and $\beta(\eps) = a_{0,2} \eps^2 + \ocal(\eps^3)$.
The fixed points $w_\pm(\eps)$ of $p_\eps$ are the zeros of $x \mapsto x^2+ \alpha(\eps) x + \beta(\eps)$.
It follows that $w_\pm(\eps) =  \frac{-\alpha(\eps)\pm \sqrt{\alpha(\eps)^2-4\beta(\eps)}}{2}$.
In particular, 
$$w_\pm(\eps) = \frac{-a_{1,1}\pm \sqrt{a_{1,1}^2-4a_{0,2}}}{2} \eps + \ocal(\eps^{3/2}),$$
and $w_\pm$ are complex differentiable at $\eps = 0$,
with $w_\pm'(0) = \frac{-a_{1,1}\pm \sqrt{a_{1,1}^2-4a_{0,2}}}{2}$.
	
It is not yet clear that $w_+$ and $w_-$ are holomorphic near $0$; note however that if 
$$w_+'(0) - w_-'(0) = \sqrt{a_{1,1}^2-4a_{0,2}} \neq 0$$ 
then $\eps \mapsto  {\alpha(\eps)^2-4\beta(\eps)}$ vanishes exactly at order 2 and in this case the two fixed points $w_\pm(\eps)$ depend holomorphically on $\eps$.
	
Next, we write $p_\eps(x) = x + (x-w_-(\eps))(x-w_+(\eps)) (1 + \ocal(x, \eps))$, $\widetilde  b_\eps(x,y) = b_1 \eps + b_2 x + b_3 y + \ocal_2(x,y,\eps)$ and 
$\widetilde  c_\eps(x,y) = c_1 \eps + c_2 x + c_3 y + \ocal_2(x,y,\eps).$ Note that $c_2 = \alpha + 1$, where $\alpha$ is the director of $v$, so $c_2 \neq 0$ by our assumptions. Let 
$$H_\eps(X,Y) := \frac{\widetilde  g_\eps(\eps X, \eps Y) - (\eps X, \eps Y)}{\eps^2}$$ 
so that $H_\eps(X,Y) = (0,0)$ if and only if 
$(\eps X, \eps Y)$ is a fixed point of $\widetilde  g_\eps$.	Then 
$$H_\eps(X,Y) = \left( (X-w_-'(0)) (X-w_+'(0)) + Y(b_1 + b_2 X + b_3 Y ),   Y(c_1 + c_2 X + c_3 Y )   \right) + \ocal(\eps),$$
so as $\eps \to 0$, the map $H_\eps$ converges locally uniformly to 
$$H(X,Y) = \left( (X-w_-'(0)) (X-w_+'(0)) + Y(b_1 + b_2 X + b_3 Y ),   Y(c_1 + c_2 X + c_3 Y )   \right).$$
Since $c_2\neq 0$, it is then straightforward to check that the set 
$H^{-1}(0,0)$ is finite and contains 4 elements counted with multiplicity. 	Moreover, $H(X,0)=(0,0)$ if and only if $X = w_\pm'(0)$.
	
Since proper intersections of analytic sets persist under perturbations, for all $\eps$ small enough the set $H_\eps^{-1}(0,0)$ is has the same 
cardinality as $H^{-1}(0,0)$ and its elements are close to those of 
$H^{-1}(0,0)$. Therefore, for small $\eps$ the fixed points of the  map $\widetilde  g_\eps$ are of the form 
$\eps v_i + \ocal(\eps^2)$, where  $H(v_i) = (0,0)$.
	
In particular, $\widetilde  g_\eps$ has exactly two fixed points $\widetilde  z_i(\eps)$ (with $1 \leq i \leq 2$)
which are asymptotically tangent to $v = (1,0)$ and moreover, $\widetilde  z_i(\eps) = (w_\pm(\eps),0) + \ocal(\eps^2)$. Finally, the assertion that $w_+'(0) \neq w_-'(0)$ follows from the fact that $\widetilde  z_1'(0) \neq \widetilde  z_2'(0)$.
In particular, $\eps \mapsto w_\pm(\eps)$ are indeed holomorphic. 
\end{proof}

\begin{rem}
As we mentioned in Remark~\ref{rem:H_3'}, in our results hypothesis $(H_3)$ can be replaced by the weaker assumption $(H_3')$. To show this, it suffices to note that if we impose hypothesis $(H_3')$ then the jet or order $m+1$ of $p_{\eps}(x)$ has two fixed points $\hat w_{\pm}(\eps)$ with $\hat w_+'(0)\neq\hat w_-'(0)$, so $p_\eps(x)$ has two fixed points $w_{\pm}(\eps)$ with $w_+'(0)\neq w_-'(0)$ and then Lemma~\ref{lem:fptgtv} also holds with the same proof.
\end{rem}

\begin{lem}\label{lem:affinechange+d_m+3}
Up to replacing $\eps$ by $\widetilde  \eps:=\mu \eps$ for some $\mu \neq 0$,
there is a family of polynomial automorphisms $\Theta_\eps : \C^2 \to \C^2$, depending holomorphically on $\eps$ near $0$ and such that  $\Theta_0=\id$, such that the maps $g_{\eps}:=\Theta_\eps\circ \widetilde  g_{\eps}\circ \Theta_\eps^{-1}$ have the form 
$$
g_\eps(x,y) = (x + (x^2+\eps^2)a_\eps(x)+ y b_\eps(x,y),  y + y c_\eps(x,y) + d_\eps(x) )
$$
with $a_\eps, b_\eps, c_\eps$ and $d_\eps$ depending holomorphically on $\eps$ and	
\begin{enumerate}
\item $a_{0}(0)=1$, $b_0(0,0)=0$
\item $c_{\eps}(x,y)= \eta x + q  \eps +c y + \ocal_2(x,y,\eps)$, with $\re\eta>3$, $q, c\in\C$ and $d_{\eps}(x) =\ocal(x^{m+3})+ \eps\ocal_{m+1}(x,\eps)$ where $m =\lfloor\re\eta\rfloor$.
\end{enumerate}
\end{lem}

\begin{proof}
By Lemma~\ref{lem:fptgtv}, the fixed points $w_{\pm}(\eps)$ of $p_{\eps}$ depend holomorphically on $\eps$ and satisfy $w_+'(0)\neq w_-'(0)$. Therefore, there exists $\mu \in \C^*$ such that if we replace $\eps$ by $\widetilde  \eps:=\mu \eps$ then $\frac{d}{d\widetilde  \eps}_{|\widetilde  \eps=0} (w_+(\widetilde  \eps)-w_-(\widetilde  \eps) )= 2i$. By an abuse of notation, we will still denote $\widetilde  \eps$ by $\eps$: in other words, we assume from now on that  $w_+'(\eps)-w_-'(\eps)) = 2i$.
	
For $\eps \neq 0$, let $N_\eps$ be the unique affine automorphism of $\C$ mapping $w_\pm(\eps)$ to $\pm i \eps$ (note that $N_\eps$ is well-defined for $\eps$ small since $w_+(\eps) \neq w_-(\eps)$). More explicitly, 
$$N_\eps(x) = \frac{2i\eps}{w_+(\eps)-w_-(\eps)} (x- w_+(\eps)) + i \eps$$ 
and $\eps \mapsto N_\eps$ extends holomorphically to a neighborhood of $0$  with 
$$N_0(x) := \frac{2i}{w_+'(0) - w_-'(0)} x = x.$$
Let $M_\eps(x,y):=(N_\eps(x),y)$ and $\widehat g_\eps:=M_\eps \circ \widetilde  g_\eps \circ M_\eps^{-1}$. A direct computation using expression \eqref{eq:tildeg} shows that $\widehat g_\eps$ is of the form
$$\widehat g_\eps(x,y) = \left(N_\eps \circ p_\eps \circ N_\eps^{-1}(x) + y \widehat b_\eps(x,y), y (1 + \widehat c_\eps(x,y)) + \widehat d_{\eps}(x) \right)$$
with $\widehat b_{\eps}, \widehat c_{\eps}, \widehat d_{\eps}$ depending holomorphically on $\eps$ and $\widehat d_{\eps}(x)=\ocal(x^{m+2})$. Moreover, by construction,  $N_\eps \circ p_\eps \circ N_\eps^{-1}$ has fixed points at $\pm i\eps$ so we can write $N_\eps \circ p_\eps \circ N_\eps^{-1}(x) = x + (x^2+\eps^2) \widehat a_\eps(x)$, with $\widehat a_{\eps}$ depending holomorphically on $\eps$.
Therefore
\begin{equation*}
\widehat g_\eps(x,y) = \left(x + (x^2+\eps^2)\widehat a_\eps(x)+ y \widehat b_\eps(x,y),  y + y \widehat c_\eps(x,y) +  \widehat d_{\eps}(x)\right).
\end{equation*}
Since $N_0=\id$ and $p_0(x) = x + x^2 +\ocal(x^3)$,  we have $\widehat a_0(0)=1$. And since $M_0=\id$, we also have 
$\widehat b_0(0,0) = \widehat c_0(0,0)=1$. Write 
\begin{align*}
\widehat c_{\eps}(x,y)=\eta x+q\eps+\ocal(x^2,y,x\eps,\eps^2);\quad \widehat d_{\eps}(x)=dx^{m+2}+\ocal(x^{m+3})+\eps\ocal_{m+1}(x,\eps)
\end{align*} 
for some $q, d\in\C$, where $\eta=\alpha+1$, so $\re\eta>3$ and $m=\lfloor \re\eta\rfloor$. Now, we consider the polynomial change of coordinates $\Psi_\eps$ given by
$$\Psi_\eps(x,y)=\left(x,y-\frac d{m+1-\eta}x^{m-1}(x^2+\eps^2)\right).$$
Let $g_{\eps}:=\Psi_\eps\circ \widehat g_{\eps}\circ \Psi_\eps^{-1}$. If we denote $(x_1,v_1):=g_{\eps}(x,v)$, we have that
$$x_1=x+(x^2+{\eps^2}) a_{\eps}(x) + v b_{\eps}(x,v)$$
for some $a_{\eps}$ and $b_{\eps}$ depending holomorphically on $\eps$ and such that $a_0(0)=\widehat a_0(0)=1$ and $b_0(0,0)=\widehat b_0(0,0)=0$.
Now, denote $\ell_{\eps}(x,y)=(x^2+\eps^2)\widehat a_{\eps}(x)+y\widehat b_{\eps}(x,y)$, so that $x_1=x+\ell_{\eps}(x,y)$. Given $k\in\N$, we have that
$$x_1^{k}=x^{k}+kx^{k-1}\ell_{\eps}(x,y)+\sum_{j=2}^{k}\binom{k}{j}x^{k-j}\ell_{\eps}(x,y)^j.$$
Since $\ell_{\eps}(x,y)=x^2+\eps^2+\ocal_3(x,\eps)+y\ocal(x,y,\eps)$, we have
$$x^{k-1}\ell_{\eps}(x,y)=x^{k+1}+x^{k-1}\eps^2+\ocal_{k+2}(x,\eps)+y\ocal_{k}(x,y,\eps)$$
and since $\ell_{\eps}(x,y)=\ocal_2(x,y,\eps)$ we have 
$$x^{k-j}\ell_{\eps}(x,y)^j=\ocal_{k+2}(x,y,\eps)$$
for all $2\le j\le k$. Therefore
\begin{align*}
x_1^{k}&=x^{k}+kx^{k+1}+kx^{k-1}\eps^2+\ocal_{k+2}(x,\eps)+y\ocal_{k}(x,y,\eps)+\ocal_{k+2}(x,y,\eps),
\end{align*}
so 
$$x_1^{m+1}+x_1^{m-1}\eps^2=x^{m+1}+(m+1)x^{m+2}+x^{m-1}\eps^2+\ocal(x^{m+3})+\eps\ocal_{m+1}(x,\eps)+y\ocal_{m+1}(x,y,\eps).$$
Then we have,
\begin{align*}
v_1&= y\left(1+\eta x+q{\eps}+\ocal\left(x^2,y,x\eps,\eps^2\right)\right)+d x^{m+2}+\ocal(x^{m+3})+\eps\ocal_{m+1}(x,\eps)+y\ocal_{m+1}(x,y,\eps)\\
&\qquad -\frac{d}{m+1-\eta}\left(x^{m+1}+(m+1)x^{m+2}+x^{m-1}\eps^2\right)\\
&=\left(v+\frac{d}{m+1-\eta}\left(x^{m+1}+x^{m-1}\eps^2\right)\right)\left(1+\eta x+q{\eps}+\ocal\left(x^2,v,x\eps,\eps^2\right)\right)+d x^{m+2}\\
&\qquad -\frac{d}{m+1-\eta}\left(x^{m+1}+(m+1)x^{m+2}+x^{m-1}\eps^2\right)+\ocal(x^{m+3})+\eps\ocal_{m+1}\left(x,\eps\right)\\
&=v\left(1+\eta x+q{\eps}+\ocal\left(x^2,v,x\eps,\eps^2\right)\right)+\frac{d}{m+1-\eta}\left(x^{m+1}+\eta x^{m+2}+x^{m-1}\eps^2\right)+d x^{m+2}\\
&\qquad -\frac{d}{m+1-\eta}\left(x^{m+1}+(m+1)x^{m+2}+x^{m-1}\eps^2\right)+\ocal(x^{m+3})+\eps\ocal_{m+1}\left(x,\eps\right)\\
&=v\left(1+\eta x+q{\eps}+\ocal\left(x^2,v,x\eps,\eps^2\right)\right)+\ocal(x^{m+3})+\eps\ocal_{m+1}\left(x,\eps\right)
\end{align*}
and the Lemma is proved taking $\Theta_\eps:=\Psi_\eps\circ M_\eps$.
\end{proof}

We will also need the following result.

\begin{lem}\label{lem:pertid}
Let $\eps\mapsto A(\eps)$ be a holomorphic map from $\D$ to $\mcal_2(\C)$. Assume that $A(0)=\id$, and $A'(0)=\begin{pmatrix}
u  & w \\
0 & v
\end{pmatrix}$.
Then, for $\eps \neq 0$ small enough, $A(\eps)$ has two  eigenvalues $\lambda_1(\eps) = 1+u \eps+\ocal(\eps^2)$, $\lambda_2(\eps) = 1+v \eps + \ocal(\eps^2)$. Moreover, if $u \neq v$ then the eigenspace associated to $\lambda_1(\eps)$ tends to $\C (1,0)$ as $\eps \to 0$.
\end{lem}

Observe that the condition $u \neq v$ is necessary: the matrices $A(\eps) = \begin{pmatrix}
1+u\eps & \eps^2 \\
\eps^2 & 1+u\eps
\end{pmatrix}$ have eigenvectors $(1,1)$ and $(1,-1)$ for all $\eps \neq 0$.

\begin{proof}
Write $A(\eps)=(a_{ij}(\eps))_{1 \leq i,j\leq 2}$, and let 
$$p(\eps,\mu):=\frac{1}{\eps^2}\begin{vmatrix}
a_{11}(\eps)-1-\mu \eps & a_{12}(\eps)\\
a_{21}(\eps) & a_{22}(\eps)-1-\mu\eps 
\end{vmatrix} = \frac{\chi_{A(\eps)}(1+\mu\eps)}{\eps^2}.$$
Then $p(\eps,\mu) = \begin{vmatrix}
u-\mu & w \\
0 & v-\mu
\end{vmatrix} + \ocal(\eps)$; this proves that for $\eps \neq 0$ small enough the eigenvalues of $A(\eps)$ satisfy $\lambda_1(\eps) = 1+u \eps+\ocal(\eps^2)$, $\lambda_2(\eps) = 1+v \eps + \ocal(\eps^2)$. In particular, if $u\neq v$ then for $\eps \neq 0$ small enough they are simple.

Assume now that $u\neq v$ and let us prove the assertion about the eigenspace associated to $\lambda_1(\eps)$. Let $X(\eps)$ be a family of eigenvectors associated to $\lambda_1(\eps)=1+u \eps + \ocal(\eps^2)$ for $\eps \neq 0$, and let $\eps_n \to 0$ be any sequence. Assume without loss of generality that $\| X(\eps_n)\|=1$ for all $n \in \N$.
For every $n$, we have 
$$A(\eps_n) X(\eps_n) = \lambda_1(\eps_n) X(\eps_n)$$
and then
$$\frac{A(\eps_n)-\id}{\eps_n} X( \eps_n) = u X(\eps_n) + \ocal(\eps_n);$$
therefore any adherence value $X \in \C^2$  of the sequence $(X(\eps_n))_{n \geq 0}$ satisfies $A'(0) X = u X$. In particular, $X$ is in $\C(1,0)$. 
Therefore, $\limn [X(\eps_n)] = [1:0]$.
\end{proof}

We can now conclude the proof of Proposition~\ref{prop:2in1intro} and Theorem~\ref{th:v2}. The first statement of Proposition~\ref{prop:2in1intro} follows from Lemma~\ref{lem:fptgtv}. Let us compute the differential of $g_\eps$ at the fixed points $\widehat z_1(\eps)=\Theta_{\eps}(\widetilde  z_1(\eps))=(i\eps,0) + \ocal(\eps^2)$ and $\widehat z_2(\eps)=\Theta_{\eps}(\widetilde  z_2(\eps))=(-i\eps,0)+\ocal(\eps^2)$. 
	
We have
$$\jac\, g_\eps(\widehat z_1(\eps))=\begin{pmatrix}
1+2i\eps+\ocal(\eps^2)  &    b_\eps(i\eps,0) + \ocal(\eps^2) \\
\ocal(\eps^2) & 1+q \eps + i \eta \eps + \ocal(\eps^2)
\end{pmatrix}.$$	
and 
$$\jac\, g_\eps(\widehat z_2(\eps))=\begin{pmatrix}
1-2i\eps+\ocal(\eps^2)  &    b_\eps(-i\eps,0) + \ocal(\eps^2) \\
\ocal(\eps^2) & 1+q \eps - i \eta \eps + \ocal(\eps^2)
\end{pmatrix}.$$	
Then, by Lemma~\ref{lem:pertid},  $\jac\, g_\eps(\widehat z_1(\eps))$ has eigenvalues 
$$\rho_T^1(\eps)=1+2i\eps+\ocal(\eps^2), \quad \rho_N^1(\eps)=1+(q+i\eta)\eps+\ocal(\eps^2)$$ 
and $\jac\, g_\eps(\widehat z_2(\eps))$ has eigenvalues 
$$\rho_T^2(\eps)=1-2i\eps+\ocal(\eps^2), \quad \rho_N^2(\eps)=1+(q-i\eta)\eps+\ocal(\eps^2).$$ 
If $(\rho_T^i)'(0)\neq (\rho_N^i)'(0)$ then again by Lemma~\ref{lem:pertid} the eigenspace associated to $\rho_T^i(\eps)$ tends to $\C (1,0)$ as $\eps\to0$. Moreover, this condition happens for at least one of the fixed points, since either $q+i\eta \neq 2i$ or $q-i\eta \neq -2i$ (or both). This finishes the proof of Proposition~\ref{prop:2in1intro}.

Let us now prove Theorem~\ref{th:v2}. Observe first that 
$$\lim_{\eps\to0} \frac1{\eps}\frac{\rho_N^1(\eps)+\rho_N^2(\eps)-2}{\rho_T^1(\eps)+\rho_T^2(\eps)}=\lim_{\eps\to0} \frac1{\eps}\frac{2q\eps+\ocal(\eps^2)}{2+\ocal(\eps^2)}=q.$$
Now, consider a sequence $\eps_n \to 0$ be such that $\frac{2i\pi}{\rho_T^i(\eps_n)-1} = n-\sigma+o(1)$ for some $1\le i\le 2$ and for some constant $\sigma \in \C$. Up to replacing $\eps$ by $-\eps$ in the family $(f_\eps)$, we can assume without loss of generality that $\rho_T^i=\rho_T^1$. Then, writing $\rho_T^1(\eps)=1+2i\eps_n+\beta\eps_n^2+\ocal(\eps_n^3)$ (where $\beta$ depends only on the family $(f_{\eps_n})$) we have
$$\frac{2i\pi}{2i\eps_n+\beta\eps_n^2+\ocal(\eps_n^3)} = n-\sigma + o(1),$$
and therefore $n-\pi/\eps_n = \sigma+\sigma_0 + o(1)$, with $\sigma_0=i\pi\beta /2$. By Theorem~\ref{th:explicit} we have that for any compact set $K$ contained in the parabolic basin of $g_0$ associated to $(1,0)$ there exists $N \in \N$ such that 
$$g_{\eps_n}^{n-N} \to \lcal_{\sigma+\sigma_0-N,q}^{(g_0)}$$
uniformly on $K$, where $\lcal_{\sigma+\sigma_0-N,q}^{(g_0)}:= (\Phi_{(g_0)}^o)^{-1} \circ A_{\sigma+\sigma_0-N,q} \circ \Phi_{(g_0)}^\iota$ and $\Phi_{(g_0)}^\iota$, $\Phi_{(g_0)}^o$ denote respectively the incoming and outgoing Fatou coordinates for $g_0$. Since 
$f_{\eps}=\Psi^{-1}\circ\Theta_{\eps}^{-1}\circ g_{\eps}\circ\Theta_{\eps}\circ\Psi$ and $\Theta_0=\id$, we deduce that 
$$\limn f_{\eps_n}^{n-N} = \Psi^{-1} \circ \lcal_{\sigma+\sigma_0-N,q}^{(g_0)} \circ \Psi.$$
It is straightforward to see that $\Phi_{(f_0)}^{\iota,o}:= \Phi_{(g_0)}^{\iota,o} \circ \Psi$ are respectively the incoming and outgoing Fatou coordinates for $f_0$. It follows that for any compact $K$ contained in the parabolic basin of $f_0$ associated to $v$ there exists $N$ such that $\limn f_{\eps_n}^{n-N} = (\Phi_{(f_0)}^o)^{-1} \circ A_{\sigma+\sigma_0-N,q} \circ \Phi_{(f_0)}^\iota$ and Theorem~\ref{th:v2} is proved.

\section{Asymptotics of Fatou coordinates}\label{sec:fatoucoord}

Consider a family $(g_{\eps})$ as in Theorem~\ref{th:explicit}, so $g:=g_0$ has the form
$$g(x,y) = (x + x^2a_0(x)+ y b_0(x,y),  y + y c_0(x,y) + d_0(x) )$$
with $a_{0}(0)=1$, $b_0(0,0)=0$, $c_0(x,y)= \eta x +c y + \ocal_2(x,y,\eps)$ and $d_0(x) =\ocal(x^{m+3})$
with $\rho=\re\eta>3$ and $m=\lfloor \rho\rfloor$. Then, we can write
$$g(x,y)=\left(x+x^2+ax^3+\ocal\left(x^4,xy,y^2\right), y+\eta xy + \ocal\left(x^2y,y^2,x^{m+3}\right)\right).$$
In this section we recall the construction and asymptotics of Fatou coordinates for $g$. Although these results are essentially contained in \cite{hakim1997transformations} (see also \cite{lopez2025flower}) we provide all the proofs for the sake of completeness.

In the following, $\log$ refers to the principal branch of the complex logarithm, defined on $\C \setminus (-\infty, 0]$. The expression  $\left(-\frac{1}{x} \right)^\eta$ 
(defined in particular when $\re x<0)$ means by definition $\exp\left(\eta \log \left(-\frac{1}{x}\right)\right)$.

Denote, for any $C>0$ and any $r>0$
$$P^\iota(r,C)=\left\{(x,y)\in\C^2:|x+r|<r,\left|\frac y{(-x)^\eta}\right|<C\right\}.$$
Although we will not mention it explicitly, in the following computations we always assume that $r$ is small enough such that $P^\iota(r,C)\subset U$, where $U$ is the domain of definition of $(g_{\eps})$.
\begin{lem}\label{lem:petalsforg}
For any $C>0$ there exists $r_0(C)>0$ such that if $0<r\le r_0(C)$ and $(x,y)\in P^\iota(r,C)$ then
$$g^n(x,y)\in P^\iota(r,C+1)$$
for every $n\ge0$. Moreover, if $(x,y)\in P^\iota(r,C)$ and we denote $(x_n,y_n):=g^n(x,y)$ we have that 
\begin{equation}\label{eq:-1/x1}
-\frac1{x_1}=-\frac1x+1+(a-1)x+\ocal\left(x^2\right), \quad \frac{y_1}{(-x_1)^\eta}=\frac{y}{(-x)^\eta}+\ocal\left(x^2\right)
\end{equation}
as $(x,y)\to0$ (the constants in $\ocal$ being allowed to depend on $C$),
$$\lim_{n\to\infty} \dfrac{-1}{nx_n}=1\quad\text{ and }\quad |x_n|\le\left(\re\left(-\frac1{x}\right)+\frac n2\right)^{-1} \text{ for any } n\ge1.$$ 
\end{lem}
\begin{proof}
Fix $C>0$ and set $(x_1,y_1)=g(x,y)$. Let $0<r_0<r_1\le 1$ and consider $(x,y)\in P^\iota(r,C)$. By the expression of $g$, if $|y(-x)^{-\eta}|<C+1$ (so in particular $y=\ocal(x^3)$ since $\rho>3$) we have that $x_1=x+x^2+ax^3+\ocal(x^4)$ and $y_1=y(1+\eta x+\ocal(x^2))+\ocal(x^{m+3})$ (where the constants in $\ocal(x^4)$ and $\ocal(x^2)$ depend on $C$), so
\begin{equation*}
-\frac1{x_1}=-\frac1{x\left(1+x+ax^2+\ocal\left(x^3\right)\right)}=-\frac1x+1+(a-1)x+\ocal\left(x^2\right)
\end{equation*}
and, as long as $\re x<0$ and $|x|$ is small enough so that $(-x_1)^\eta$ is defined,
\begin{align*}
\frac{y_1}{(-x_1)^\eta}
&=\frac{y\left(1+\eta x+\ocal(x^2)\right)+\ocal(x^{m+3})}{(-x)^\eta\left(1+x+\ocal(x^2)\right)^\eta}\\
&=\frac{y}{(-x)^\eta}\left(1+\ocal(x^2)\right)+\ocal(x^{m+3-\rho})=\frac{y}{(-x)^\eta}+\ocal\left(x^2\right),
\end{align*}
(using in the last equality that $m+3-\rho>2$ by definition of $m$) so there exists a constant $\widetilde  C>0$, depending on $C$, such that
$$\left|\frac{y_1}{(-x_1)^\eta}\right|\le \left|\frac{y}{(-x)^\eta}\right|+\widetilde  C|x|^2.$$
By the expression of $-1/x_1$, we can choose $r_0=r_0(C)>0$ small enough such that if $0<r\le r_0$ and $|x+r|<r$ (so $\re(-1/x)>(2r)^{-1}$) then 
$$\re\left(-\frac1{x_1}\right)\ge \re\left(-\frac1x\right)+\frac12,$$
so in particular $|x_1+r|<r$ and moreover 
$$|x_1|\le \left(\re\left(-\frac1{x}\right)+\frac12\right)^{-1}.$$ 
Up to decreasing $r_0$ if necessary, we can also assume that 
$$\widetilde C\sum_{n\ge0}\left(\frac1{2r_0}+\frac n2\right)^{-2}<1.$$

Take $(x,y)\in P^\iota(r,C)$ with $0<r\le r_0$. Since $P^\iota(r,C)\subset P^\iota(r,C+1)$, we have by the previous computation that $|x_1+r|<r$ (so in particular $(-x_1)^\eta$ is defined) and 
$$\left|y_1(-x_1)^{-\eta}\right|<C+\widetilde C|x|^2<C+\widetilde C(2r)^2<C+1,$$ 
so $(x_1,y_1)\in P^\iota(r,C+1)$. Arguing recursively, we obtain that $|x_n+r|<r$,
\begin{equation*}
|x_n|\le\left(\re\left(-\frac1{x}\right)+\frac n2\right)^{-1}
\end{equation*}
and  
\begin{align*}
\left|\frac{y_n}{(-x_n)^{\eta}}\right|
&\le\left|\frac{y}{(-x)^{\eta}}\right|+\widetilde C\sum_{j=0}^{n-1}|x_j|^2< C+\widetilde C\sum_{j=0}^{n-1}\left(\re\left(-\frac1{x}\right)+\frac j2\right)^{-2}\\
&\le C+\widetilde C\sum_{j=0}^{n-1}\left(\frac1{2r}+\frac j2\right)^{-2} <C+1,
\end{align*}
for every $n\ge1$, where $(x_n,y_n)=g^n(x,y)$, so $(x_n,y_n)\in P^\iota(r,C+1)$ for every $n\ge1$. Moreover, since
$$-\frac{1}{x_1}=-\frac{1}{x}+1+h(x,y)$$
where $h(x,y)=\ocal(x)$, we have that
$$-\frac{1}{x_n}=-\frac{1}{x}+n+\sum_{j=0}^{n-1}h(x_j,y_j)$$
and therefore $\frac{-1}{nx_n}\to 1$ as $n\to\infty$.
\end{proof}

Our next goal is to prove the following Propositions:

\begin{prop}\label{prop:incomfatou}
For any $C>0$ there exists $0<r^\iota(C)\le r_0(C)$ such that for any $0<r\le r^\iota(C)$ there exists a holomorphic univalent map $\Phi^\iota: P^\iota(r,C)\to\C^2$ which is an incoming Fatou coordinate for $g$ (i.e., $\Phi^\iota\circ g=\Phi^\iota+(1,0)$) and satisfies 
$$\Phi^\iota(x,y)=\left(-\frac{1}{x} + (1-a) \log (-x) + o\left(1\right), \frac{y}{(-x)^\eta} + o\left(1\right)\right),$$
as $\re(-1/x)\to+\infty$ inside $P^\iota(r,C)$.
Moreover, 
$$\{ X \in \H_{r^{-1}}:|\im X|<2|\re X|\} \times \D(0,C-1)\subset \Phi^\iota(P^\iota(r,C)),$$
where $\H_{t}:=\left\{X\in \C: \re X>t\right\}$.
\end{prop}

\begin{prop}\label{prop:outgofatou}
For any $C>0$ there exists $r^o(C)>0$ such that for any $0<r\le r^o(C)$ there exists a holomorphic univalent map $\Phi^o: P^o(r,C)\to\C^2$, where
$$P^o(r,C)=\left\{(x,y)\in\C^2:|x-r|<r,\left|\frac y{x^\eta}\right|<C\right\},$$ 
which is an outgoing Fatou coordinate for $g$ (i.e., $\Phi^o\circ g=\Phi^o+(1,0)$) and satisfies 
$$\Phi^o(x,y)=\left(-\frac{1}{x} + (1-a) \log x + o\left(1\right), \frac{y}{x^\eta} + o\left(1\right)\right),$$
as $\re(-1/x)\to-\infty$ inside $P^o(r,C)$.
Moreover,
$$-\{ X \in \H_{r^{-1}}:|\im X|<2|\re X|\} \times \D(0,C-1)\subset \Phi^o(P^o(r,C)).$$
\end{prop}

Set
$$\Phi_0^\iota(x,y):= \left( -\frac{1}{x} , \frac{y}{(-x)^\eta} \right).$$
It is straightforward to check that $\Phi_0^\iota$ is well-defined and univalent on $\{(x,y) \in \C^2: \re x<0\}$ and that 
\begin{equation*}
\left(\Phi_0^\iota \right)^{-1}(X,Y) = \left(- \frac{1}{X}, \frac{Y}{X^\eta} \right). 
\end{equation*}
By Lemma~\ref{lem:petalsforg}, for any $C>0$ there exists $r_0(C)>0$ such that for any $0<r\le r_0(C)$ the map 
$$G:=\Phi_0^\iota \circ g \circ \left(\Phi_0^\iota \right)^{-1}$$
is well-defined on 
$$\Phi_0^\iota\left(P^\iota(r,C)\right)=\H_{(2r)^{-1}}\times\D(0,C),$$
where $\H_{(2r)^{-1}}=\left\{X\in\C:\re X>(2r)^{-1}\right\}$, and $G^n(X,Y)\in\Phi_0^\iota\left(P^\iota(r,C+1)\right)$ for any $n\ge 1$ and for any $(X,Y)\in \H_{(2r)^{-1}}\times\D(0,C)$. Moreover, using \eqref{eq:-1/x1} we have that
\begin{equation}\label{eq:expressionG}
G(X,Y) = \left(X+1+\frac{1-a}{X}  +  \ocal\left(\frac{1}{X^{2}} \right), Y  +  \ocal\left(\frac{1}{X^{2}}\right) \right)
\end{equation}
in $\H_{(2r)^{-1}}\times\D(0,C)$ as $\re X \to +\infty$ (the constants in $\ocal$ being allowed to depend on $C$).

\begin{lem}\label{lem:constructphi2}
For any $C>0$ and any $0<r\le r_0(C)$, where $r_0(C)$ is given by Lemma~\ref{lem:petalsforg}, the map
$$\psi(X,Y):=\lim_{n\to\infty} Y_n,$$ 
where $(X_n,Y_n)=G^n(X,Y)$, is well-defined in $\H_{(2r)^{-1}}\times\D(0,C)$, satisfies $\psi\circ G=\psi$ and has the form $\psi(X,Y)=Y+o(1)$ as $\re X\to+\infty$. Moreover, the map $\Phi_1^\iota:\H_{(2r)^{-1}}\times\D(0,C)\to\C^2$ defined by
$$\Phi_1^\iota(X,Y):=\left(X,\psi(X,Y)\right)$$
is injective. 
\end{lem}
\begin{proof}
By \eqref{eq:expressionG}, we have that
$Y_1=Y+k(X,Y)$ for some holomorphic map $k(X,Y)=\ocal(X^{-2})$ and for every $(X,Y)\in\H_{(2r)^{-1}}\times\D(0,C)$. Therefore, for every $n$ we have that
$$Y_n=Y+\sum_{j=0}^{n-1}k(X_j,Y_j).$$
Using the bound $|X_j|^{-1}\le \left(\re X+j/2\right)^{-1}\le 2/j$ from Lemma~\ref{lem:petalsforg} we have that the series
$\sum_{j=0}^\infty k(X_j,Y_j)$ is uniformly convergent in $\H_{(2r)^{-1}}\times\D(0,C)$, so $Y_n$ converges uniformly to a holomorphic function 
$$\psi(X,Y)=Y+v(X,Y)$$
where $v(X,Y)=\sum_{j=0}^\infty k(X_j,Y_j).$ The invariance of $\psi$ is immediate, by construction. Moreover, since
$$\sum_{j=0}^\infty \frac1{|X_j|^2}\le \frac1{(\re X)^2}+\int_0^\infty \frac1{(\re X+t/2)^2}dt=\frac1{(\re X)^2}+\frac2{\re X}$$
we have that $v(X,Y)=o(1)$ as $\re X\to+\infty$. Finally, since $Y_n$ is injective when $X$ is fixed and $\psi$ is not constant, we obtain that $\Phi_1^\iota$ is injective.
\end{proof}

By Lemma~\ref{lem:constructphi2}, for any $C>0$ and any $0<r\le r_0(C)$ the map 
$$\widetilde G:=\Phi_1^\iota \circ G \circ \left(\Phi_1^\iota\right)^{-1}$$
is well-defined on $\Phi_1^\iota\left(\H_{(2r)^{-1}}\times\D(0,C)\right)$ and $\widetilde G^n(X,Y)\in\Phi_1^\iota\left(\H_{(2r)^{-1}}\times\D(0,C+1)\right)$ for any $n\ge 1$ and for any $(X,Y)\in \Phi_1^\iota\left(\H_{(2r)^{-1}}\times\D(0,C)\right)$.
Moreover, using \eqref{eq:expressionG} and Lemma~\ref{lem:constructphi2} we have that
\begin{equation}\label{eq:G_2}
\widetilde G(X,Y) = \left(X+1+\frac{1-a}{X}  +  \ocal\left(\frac{1}{X^{2}} \right), Y \right)
\end{equation}
in $\Phi_1^\iota\left(\H_{(2r)^{-1}}\times\D(0,C)\right)$ as $\re X \to +\infty$. Observe also that since $\psi(X,Y)=Y+o(1)$ as $\re X\to+\infty$, there exists $0< r_1(C)\le r_0(C)$ such that for every $0<r\le r_1(C)$ 
$$\Phi_1^\iota\left(\H_{(2r)^{-1}}\times\D(0,C)\right)\subset \Phi_0^\iota\left(\H_{(2r)^{-1}}\times\D(0,C+1)\right).$$

\begin{lem}\label{lem:constructphi}
For any $C>0$ there exists $r_2(C)>0$ such that for any $0<r\le r_2(C)$, the map
$$\varphi(X,Y):=\lim_{n\to\infty} \left[X_n - n - (1-a) \log n\right],$$ 
where $(X_n, Y):= \widetilde{G} ^{n}(X,Y)$, is well-defined in $\Phi_1^\iota\left(\H_{(2r)^{-1}}\times\D(0,C)\right)$, satisfies $\varphi\circ \widetilde G=\varphi+1$ and has the form $\varphi(X,Y)=X-(1-a)\log X+o(1)$ as $\re X \to +\infty$. 
Moreover, the map $\Phi_2^\iota: \Phi_1^\iota\left(\H_{(2r)^{-1}}\times\D(0,C)\right)\to\C^2$ defined by
$$\Phi_2^\iota(X,Y):=(\varphi(X,Y),Y)$$ 
is injective.
\end{lem}

\begin{proof} 
Fix $C>0$, set $r_2(C)=\min\{r_0(C+1), r_1(C)\}$ and take $0<r\le r_2(C)$. 
Thanks to \eqref{eq:G_2}, for any $(X,Y)\in\Phi_1^\iota\left(\H_{(2r)^{-1}}\times\D(0,C)\right)$ we have that
$$X_1-(1-a)\log X_1=X+1-(1-a)\log X+h(X,Y)$$
with $h(X,Y)=\ocal\left(1/X^2\right)$, so
$$X_n-(1-a)\log X_n=X+n-(1-a)\log X+\sum_{j=0}^{n-1}h(X_j,Y).$$
Since $\Phi_1^\iota\left(\H_{(2r)^{-1}}\times\D(0,C)\right)\subset \Phi_0^\iota\left(\H_{(2r)^{-1}}\times\D(0,C+1)\right)$ and $0<r\le r_0(C+1)$, by  Lemma~\ref{lem:petalsforg} we have that $|X_j|^{-1}\le (\re X+j/2)^{-1}\le 2/j$ for every $j$, so the series
$\sum_{j=0}^\infty h(X_j,Y)$ is normally convergent and hence $X_n-(1-a)\log X_n-n$ converges uniformly to a holomorphic map of the form
$$X-(1-a)\log X+H(X,Y).$$ 
Moreover, arguing as in the proof of Lemma~\ref{lem:constructphi2}, $H(X,Y)=o(1)$ as $\re X\to+\infty$.
Now define 
$$\varphi_n(X,Y):=X_n - n - (1-a) \log n.$$
Then, rewriting
$$\varphi_n(X,Y)=X_n-(1-a)\log X_n-n+(1-a)\log\left(\frac {X_n}n\right)$$
and using the fact that $X_n/n\to 1$ as $n\to\infty$ by Lemma~\ref{lem:petalsforg}, we have that the sequence $\varphi_n$ converges uniformly in $\Phi_1^\iota\left(\H_{(2r)^{-1}}\times\D(0,C)\right)$ to a map $\varphi$ of the form
$$\varphi(X,Y)=X-(1-a)\log X+o(1)$$
as $\re X\to+\infty$. Since $\varphi_n \circ \widetilde G=\varphi_{n+1}+1+(1-a)\log(1+1/n)$, we obtain that $\varphi\circ\widetilde G=\varphi+1$. Injectivity of $\Phi_2^\iota$ follows from the injectivity of $\varphi_n$ when $Y$ is fixed and the fact that $\varphi$ is not constant. 
\end{proof}

We are now ready to prove Proposition \ref{prop:incomfatou}.

\begin{proof}[Proof of Proposition \ref{prop:incomfatou}]
By Lemmas~\ref{lem:constructphi2} and \ref{lem:constructphi}, for any $C>0$ and any $0<r\leq r_2(C)$ the map $\Phi_G:=\Phi_2^\iota \circ \Phi_1^\iota$ is well-defined on $\H_{(2r)^{-1}}\times\D(0,C)$ and 
$$\Phi_G(X,Y) = \left(X - (1-a) \log X , Y  \right) + o(1),$$
where the convergence in the  $o(1)$ term is uniform as $\re X \to +\infty$ in $\H_{(2r_2(C))^{-1}}\times\D(0,C)$. Moreover, $\Phi_G \circ G = \Phi_G + (1,0)$.
	
We will prove that there exists $0<r^\iota(C)\leq r_2(C)$ such that for all $0<r \leq r^\iota(C)$
\begin{equation}\label{eq:sandwichpetal1}
\{ X \in \H_{r^{-1}}:|\im X|<2 \re X \} \times \D(0,C-1)  \subset \Phi_G\left( \H_{(2r)^{-1}}  \times \D(0,C) \right) 
\end{equation}
Let us set $u:=\Phi_G - \id$ and write $u=(u_1,u_2)$. There exists $0< r^\iota(C)\le r_2(C)$ such that for all $(X,Y)\in \{X\in\H_{(2r^\iota(C))^{-1}}:|\im X|<5\re X\} \times \D(0,C)$ 
$$|u_1(X,Y)|=|(1-a)\log X+o(1)|\le \frac13\re X\quad \text{and}\quad |u_2(X,Y)|\le 1.$$
Now, consider $(X_0,Y_0)\in\{ X \in \H_{r^{-1}}:|\im X|<2\re X\} \times \D(0,C-1)$ with $0< r\le r^\iota(C)$ and let us show that there exists $(X,Y)\in \{ X \in \H_{(2r)^{-1}}:|\im X|<5\re X\} \times \D(0,C)$ with $\Phi_G(X,Y)=(X_0,Y_0)$. Let $h_0(X,Y):=(X_0,Y_0) - u(X,Y)$, so $\Phi_G(X,Y)=(X_0,Y_0)$ if and only if $(X,Y)$ is a fixed point of $h_0$, and let $K:=\overline{\D}(X_0,\re X_0/2)\times\overline{\D}(Y_0,1)$. Observe that if $(X,Y)\in K$ we have that $|Y|\le C$, $\re X\ge \re X_0/2>(2r)^{-1}$ and
$$|\im X|\le |\im X_0|+\frac{\re X_0}2< \frac52\re X_0\le 5\re X$$
so $K \subset \{ X \in \H_{(2r)^{-1}}:|\im X|<5\re X\} \times \D(0,C)$. Then if $(X_1,Y_1)=h_0(X,Y)$ we have
$$|X_1-X_0|=|u_1(X,Y)|\le \frac13\re X\le \frac12\re X_0\quad \text{ and }\quad |Y_1-Y_0|=|u_2(X,Y)|\le 1$$
so $(X_1,Y_1)\in K$. Moreover, since $u(X,Y)={((a-1)\log X,0)+}o(1)$ as $\re X \to +\infty$ in $\H_{(2r_2(C))^{-1}}\times\D(0,C)$, using Cauchy's estimate and the fact that $\log'(X)=1/X$ and reducing $r^\iota(C)$ if necessary we have that for all $(X,Y) \in \H_{(2r^\iota(C))^{-1}} \times \D(0,C)$, 
\begin{equation*}
\| \partial_X u(X,Y)\| \leq \frac{1}{2} \text{ \quad and \quad }  \| \partial_Y u(X,Y)\| \leq \frac{1}{2}.
\end{equation*} 
Then, the map $h_0:K\to K$ is contracting on $K$, so by Banach fixed point theorem $h_0$ has a fixed point in $K$ and \eqref{eq:sandwichpetal1} is proved. 

Finally, Proposition~\ref{prop:incomfatou} follows by defining $\Phi^\iota:=\Phi_G \circ \Phi_0^\iota.$
\end{proof}

\begin{proof}[Proof of Proposition \ref{prop:outgofatou}]
This follows immediately from Proposition~\ref{prop:incomfatou}. Since $g^{-1}$ is of the form
\begin{equation*}
g^{-1}(x,y) = \left( x- x^2 +(2-a) x^3+ \ocal(x^4, xy,y^2),  y - \eta xy+ \ocal(x^2y,y^2, x^{m+3})  \right),
\end{equation*}
if we take $\Psi(x,y):=(1/x,-y/x^\eta)$ we have that 
$$\Psi\circ g^{-1} \circ\Psi^{-1}(X,Y) = \left(X+1-\frac{1-a}{X}  +  \ocal\left(\frac{1}{X^{2}} \right), Y  +  \ocal\left(\frac{1}{X^{2}}\right) \right)$$
so repeating the construction above we find an incoming Fatou coordinate for $g^{-1}$ in $P^o(r, C)$ of the form 
$$\Phi^\iota_{g^{-1}}(x,y)=\left(\frac{1}{x} - (1-a) \log (x) + o\left(1\right), -\frac{y}{x^\eta} + o\left(1\right)\right),$$
as $(x,y)\to (0,0)$ inside $P^o(r,C)$. Then, to obtain Proposition~\ref{prop:outgofatou} it suffices to define $\Phi^o:= - \Phi^\iota_{g^{-1}}$. 
\end{proof}

Finally, we will also need the following asymptotic expansion for $\left({\Phi^{\iota}}\right)^{-1}$: 

\begin{lem}\label{lem:asympphi-1}
Let $C>0$,  $0<r\leq r^\iota(C)$. If $(X,Y)\in\Phi^\iota\left(P^\iota(r,C)\right)$, then
$$(\Phi^\iota)^{-1}(X,Y)= \left(-\frac{1}{X} + \ocal\left(\frac{\log X}{X^2}\right), \frac{ Y}{ X^\eta} +o\left( \frac{1}{X^\eta}\right) \right),$$
as $\re X\to+\infty$ inside $\Phi^\iota\left(P^\iota(r,C)\right)$, where the convergence in the term $o(X^{-\eta})$ and the implicit constant in the term $\ocal(X^{-2}\log X)$  are uniform on $\Phi^\iota\left(P^\iota(r,C)\right)$.
\end{lem}

\begin{proof}
Take $(x,y)\in P^\iota(r,C)$ and set $(X,Y):=\Phi^\iota(x,y)$. By Proposition~\ref{prop:incomfatou} we have
\begin{align*}
X=-\frac{1}{x} + (1-a) \log(-x) + o(1)=-\frac{1}{x} + o\left(\frac{1}{x}\right)
\end{align*}
as $\re X \to +\infty$, so $x = -1/X + o\left(1/X\right)$. Hence
\begin{align*}
X=  - \frac{1}{x} + (1-a) \log\left( \frac{1}{X} + o\left(\frac{1}{X}\right) \right) + o(1)=-\frac{1}{x} + \ocal(\log X),
\end{align*}
and so 
\begin{align*}
x &= \frac{-1}{X + \ocal\left( \log X \right)}  = -\frac{1}{X} + \ocal\left(\frac{\log X}{X^2}\right).
\end{align*}
	
Similarly:
$$Y=\frac{y}{(-x)^\eta}+o(1),$$
so 
\begin{align*}
y&= Y (-x)^\eta +  o((-x)^\eta) =  \frac{Y}{X^\eta}    \left(1+ \ocal\left(\frac{\log X}{{X}} \right) \right)^{{\eta}} +  o\left(\frac{1}{X^\eta}\right) = \frac{Y}{X^\eta} + o\left(\frac{1}{X^\eta}\right).
\end{align*}
\end{proof}

Similarly, we can compute the asymptotics of $\left(\Phi^o\right)^{-1}$: 

\begin{lem}\label{lem:asympphio-1}
Let $C>0$,  $0<r\leq r^0(C)$. If $(X,Y)\in\Phi^o\left(P^o(r,C)\right)$, then
$$(\Phi^o)^{-1}(X,Y) = \left(-\frac{1}{X} + \ocal\left(\frac{\log X}{X^2}\right), \frac{ Y}{ (-X)^\eta} +o\left( \frac{1}{X^\eta}\right) \right),$$
as $\re X\to-\infty$ inside $\Phi^o\left(P^o(r,C)\right)$, where the convergence in the term $o(X^{-\eta})$ and the implicit constant in the term $\ocal(X^{-2}\log X)$  are uniform on $\Phi^o\left(P^o(r,C)\right)$.
\end{lem}

\section{Approximate Fatou coordinates and estimation of the error terms}\label{sec:approximate}

Consider a family $(g_\eps)$  as in Theorem~\ref{th:explicit}, so 
\begin{equation}\label{eq:gsect4}
g_\eps(x,y) = (x + (x^2+\eps^2)a_\eps(x)+ y b_\eps(x,y),  y + y c_\eps(x,y) + d_\eps(x)),
\end{equation}
where $a_\eps, b_\eps, c_\eps$ and $d_\eps$ depend holomorphically on $\eps$,  $a_{0}(0)=1$, $b_0(0,0)=0$, 
$$c_{\eps}(x,y)= \eta x + q  \eps +c y + \ocal_2(x,y,\eps) \quad\text{and}\quad d_{\eps}(x) =\ocal(x^{m+3})+ \eps\ocal_{m+1}(x,\eps),$$
with $\re\eta>3$ and $m =\lfloor\re \eta\rfloor$. We consider a sequence $(\eps_n)$ such that $n-\pi/\eps_n=\sigma+o(1)$ (so $\eps_n\sim\pi/n$) and write 
\begin{align*}
&a_{\eps_n}(x)=1+a x+p\eps+\ocal(x^2,x\eps_n,\eps_n^2);\quad b_{\eps_n}(x,y)=bx+\ocal(x^2,y,\eps_n);\\
&c_{\eps_n}(x,y)=\eta x+q\eps_n+\ocal(x^2,y,x\eps_n,\eps_n^2);\quad d_{\eps_n}(x)=\ocal(x^{m+3})+\eps_n\ocal_{m+1}(x,\eps_n)
\end{align*}
for some $a, b, p \in\C$. As in \cite{BSU17semi}, up to considering the change of variables $(x,y,\eps)\mapsto (\widetilde  x,y,\widetilde  \eps)$ with
$$x=\widetilde  x(1-p\widetilde \eps);\quad \eps=\widetilde {\eps}(1-p\widetilde \eps),$$
we can assume that $p=0$.

Define
\begin{align*}
w_{\eps_n}(x):&= \frac{1}{{\eps_n}} \arctan\left(\frac{x}{{\eps_n}}\right)+\frac\pi{2{\eps_n}} + \frac{1- a}{2} \log(x^2+{\eps_n^2})\\
&=\frac1{2i{\eps_n}}\log\left(\frac{i{\eps_n}-x}{i{\eps_n}+x}\right)+\frac{\pi}{2\eps_n}+\frac{1- a}{2} \log(x^2+{\eps_n^2}) 
\end{align*}
and 
$$t_{\eps_n}(x,y)=\frac{y}{(x^2+{\eps_n^2})^{\eta/2}}.$$

We set 
$$\Phi_{\eps_n}^\iota(x,y):=(w_{\eps_n}(x),t_{\eps_n}(x,y)),  \qquad \Phi_{\eps_n}^o(x,y):= \Phi_{\eps_n}^{\iota}(x,y)- \left(\frac{\pi}{\eps_n},0 \right)$$ 
defined on $(\C \setminus L_{\eps_n}) \times \C$, where $L_{\eps_n}:= \{ it \eps_n, t \in (-\infty, -1] \cup [1,+\infty)\}$, and we fix a constant $\gamma \in \left(1/2, 2/3\right)$ chosen so that
$$\gamma \rho>2.$$

\begin{defi}\label{defi:Rn}
Set $k_n:= \lfloor n^\gamma\rfloor$. We define, for any $C>1$, the set $\mathcal{R}_n(C)\subset (\C\setminus L_{\eps_n})\times \C$ as
$$
\mathcal{R}_n(C)
:=\left\{(x,y): \re\left(\eps_n w_{\eps_n}(x)\right) \in \left[\frac{\pi k_n}{10n}, \pi-\frac{\pi k_n}{10n} \right], \left|\im (\eps_n w_{\eps_n}(x))\right|\le C\frac\pi n \text{ and } \frac{1}{C} < |t_{\eps_n}(x,y) | < C  \right\}.
$$
\end{defi}

\begin{rem}
Although we will not explicitly mention it, all the terms $o$ and $O$ appearing in the results in this Section are uniform on $\rcal_n(C)$.
\end{rem}

\begin{prop}\label{prop:phiepsinv}
For all $C>0$ and for all $n$ large enough, $(\Phi_{\eps_n}^\iota)^{-1}$ is well-defined on $\Phi_{\eps_n}^\iota\left(\mathcal{R}_n(C)\right)$
and 
$$(\Phi_{\eps_n}^\iota)^{-1}(X,Y) = \left(-\eps_n\cot\left(\eps_nX+\ocal\left(n^{-1}\log n\right)\right),Y\frac{\eps_n^\eta}{\sin^\eta\left(\eps_n X+\ocal\left(n^{-1}\log n\right)\right)}\right)$$
\end{prop}

\begin{proof}
Let 
$$\widetilde  \Phi_{\eps_n}(x,y):=\left(\frac{1}{\eps_n} \arctan\left(\frac{x}{\eps_n}\right)+ \frac{\pi}{2\eps_n}, \frac{y}{(x^2+{\eps_n^2})^{\eta/2}} \right). $$
	
We claim that the map $\widetilde  \Phi_{\eps_n}:(\C\setminus L_{\eps_n})\times\C\to\{(X,Y)\in\C^2:\re (\eps_n X)\in (0,\pi)\}$ is well-defined and bijective, and its inverse is given by 
$$(\widetilde  \Phi_{\eps_n})^{-1}(X,Y) = \left(-\eps_n \cot(\eps_n X), Y \frac{\eps_n^\eta}{\sin^\eta(\eps_n X)} \right).$$	
	
Indeed, set $(X,Y):=\widetilde  \Phi_{\eps_n}(x,y)$. We have $\eps_n X=\arctan\left(x/\eps_n\right)+\pi/2$ and therefore
$$x=\eps_n\tan\left(\eps_nX-\frac\pi2\right)=-\eps_n \cot(\eps_n X).$$ 
On the other hand, since $Y=y\left(x^2+\eps_n^2\right)^{-\eta/2}$ we get
$$y= Y (x^2+\eps_n^2)^{\eta/2}= Y \eps_n^\eta(1+\cot^2(\eps_n X))^{\eta/2}= Y \frac{\eps_n^\eta}{\sin^\eta(\eps_n X)}.$$	
This proves the claim.	
	
Now denote $\mathcal{G}_{\eps_n}:=\Phi_{\eps_n}^\iota \circ (\widetilde  \Phi_{\eps_n})^{-1}$, defined in $\{(X,Y) \in \C^2 : \re (\eps_n X) \in (0,\pi)\}$. Since
$$\Phi_{\eps_n}^\iota(x,y)=\widetilde  \Phi_{\eps_n}(x,y)+ \left(\frac{1-a}2\log(x^2+\eps_n^2),0 \right)$$
we have by the claim above that
\begin{align*}
\mathcal{G}_{\eps_n}(X,Y) 
&= \left(X,Y\right)+\left(\frac{1-a}2\log(\eps_n^2+\eps_n^2\cot^2(\eps_n X) ),0\right) \\
&= \left(X,Y\right)+\left(\frac{1-a}2 \log\left(\frac{\eps_n^2}{\sin^2(\eps_n X)}\right),0 \right). 
\end{align*}
Let us prove that $\mathcal{G}_{\eps_n}$ is invertible in  
$$
\Phi_{\eps_n}^\iota(\mathcal{R}_n(C))
=\left\{(X,Y)\in\C^2: 
\re\left(\eps_n X\right) \in \left[\frac{\pi k_n}{10n}, \pi-\frac{\pi k_n}{10n} \right], \left|\im(\eps_n X)\right|\le C\frac \pi n \text{ and } \frac{1}{C} < |Y| < C  \right\}
$$ 
for $n$ big enough. Choose $(X_0,Y_0)\in\Phi_{\eps_n}^\iota(\mathcal{R}_n(C))$ and let us show that there is a unique $(X,Y)$ with $\re(\eps_n X)\in(0,\pi)$ such that $\mathcal{G}_{\eps_n}(X,Y)=(X_0,Y_0)$. Since $\mathcal{G}_{\eps_n}(X,Y)=\left(X+u_n(X),Y\right)$, with $u_n(X)=\frac{1-a}2  \log\left(\frac{\eps_n^2}{\sin^2(\eps_n X)}\right)$, we just need to show that, for $n$ large enough, there is a unique $X$ with $\re(\eps_n X)\in(0,\pi)$ such that 
$$X+u_n(X)=X_0.$$
By Rouch\'e's theorem, it is enough to prove that $|u_n(X)|<|X-X_0|$ for every $X\in\partial Q_n$ and for $n$ big enough, where
$$Q_n=\left\{X\in\C: \re(\eps_n X)\in\left(\frac{\pi k_n}{20n}, \pi-\frac{\pi k_n}{20n}\right), |\im (\eps_n X)|<C\frac\pi n+n \right\}$$
Observe that $Q_n\subset\{X:\re(\eps_n X)\in(0,\pi)\}$ and if $\re(\eps_nX)\in(0,\pi)$ then $X\in Q_n$ for $n$ big enough. Clearly $|X-X_0|\ge \frac{\pi k_n}{20n}$ for every $X\in \partial Q_n$. Moreover, since $|\sin(\eps_n X)|\ge |\sin(\re(\eps_n X))|$ we have that $\frac{\eps_n}{\sin (\eps_n X)} = \ocal(\frac{1}{k_n})$ for every $X\in\partial Q_n$ so 
$$|u_n(X)| =\frac{|1-a|}2\left|\log \left(\frac{\eps_n^2}{\sin^2(\eps_n X)} \right) \right|=\ocal(\log k_n).$$
Hence if $n$ is big enough then $|u_n(X)|<|X-X_0|$ for every $X\in\partial Q_n$, so $\mathcal{G}_{\eps_n}$ is invertible on $\Phi_{\eps_n}^\iota\left(\mathcal{R}_n(C)\right)$, which in turn implies that 
$(\Phi_{\eps_n}^\iota)^{-1} = (\widetilde  \Phi_{\eps_n})^{-1} \circ \mathcal{G}_{\eps_n}^{-1}$ is well-defined on $\Phi_{\eps_n}^\iota\left(\mathcal{R}_n(C)\right)$. 

Let us now prove the desired estimate. Since $\mathcal{G}_{\eps_n}(X,Y)=(X+u_n(X),Y)$, we know that
$$\mathcal{G}_{\eps_n}^{-1}(X,Y) = \left(X+\ocal(\|u_n\|_\infty),Y\right).$$
Since $k_n=\lfloor n^\gamma\rfloor$, by the computations above and the maximum principle we have that $|u_n(X)| = \ocal(\log n)$ for every $(X,Y)\in\Phi_{\eps_n}^\iota(\mathcal{R}_n(C))$, so 
\begin{align*}
\mathcal{G}_{\eps_n}^{-1}(X,Y)= \left(X+\ocal( \log n), Y \right)
\end{align*}
and therefore
\begin{align*}
(\Phi_{\eps_n}^\iota)^{-1}(X,Y) = (\widetilde  \Phi_{\eps_n})^{-1} \circ \mathcal{G}_{\eps_n}^{-1}(X,Y)= (\widetilde  \Phi_{\eps_n})^{-1}(X+\ocal(\log n),Y)
\end{align*}
so the Proposition follows.
\end{proof}

\subsection{Estimation of the error terms}\label{subsec:estimationerror}
In this subsection we will give estimates for 
$$A_{\eps_n}(x,y):=w_{\eps_n}(x_1)- w_{\eps_n}(x) -1 \quad \text{ and }\quad B_{\eps_n}(x,y):=\log\frac{t_{\eps_n}(x_1,y_1)}{t_{\eps_n}(x,y)},$$ 
where $(x_1,y_1):=g_{\eps_n}(x,y)$. Recall that $n-\pi/\eps_n \to \sigma$, so $\eps_n\sim\pi/n$, $\rho=\re\eta>3$, $m=\lfloor \rho\rfloor$ and $\gamma\in(1/2,2/3)$ satisfies $\gamma\rho>2$.

\begin{lem}\label{lem:x2e2}
Take $(x,y) \in \mathcal{R}_n(C)$ and denote $(X,Y):=\Phi_{\eps_n}^\iota(x,y)$. If $n$ is big enough, we have the following estimates:
\begin{equation}\label{eq:x}
x=-\eps_n \cot(\eps_n X) + \ocal\left(\frac{\log n}{n^{2\gamma}}\right)=\ocal\left(X^{-1}\right)
\end{equation}
\begin{equation}\label{eq:1/x2+eps2}
\frac{1}{x^2+{\eps_n^2}}\sim \frac{\sin^2(\eps_n X)}{\eps_n^2} = \ocal\left(X^2\right)
\end{equation}
\begin{equation}\label{eq:y}
y \asymp\frac{\eps_n^\rho}{\sin^\rho(\eps_n X)}=\ocal\left(X^{-\rho}\right).
\end{equation}
\end{lem}

\begin{proof}
Thanks to Proposition \ref{prop:phiepsinv}, we have that if $n$ is big enough then
$$x=-\eps_n \cot\left(\eps_n X + \ocal\left(\frac{\log n}{n}\right) \right).$$
Since $(x,y)\in \mathcal{R}_n(C)$, we have that $\re(\eps_nX)\in \left[\frac{\pi k_n}{10n},\pi-\frac{\pi k_n}{10n}\right]$ and $\left|\im(\eps_n X)\right|\le C\frac\pi n$.
Then using the mean value inequality we have that for $n$ big enough
\begin{align*}
x&=-\eps_n \cot\left(\eps_n X + \ocal\left(\frac{\log n}{n}\right) \right) = - \eps_n \cot(\eps_n X) + \ocal\left(\eps_n n^{2-2\gamma} \frac{\log n}{n} \right) \\
&=-\eps_n \cot(\eps_n X) + \ocal\left(\frac{\log n}{n^{2\gamma}}\right).
\end{align*}

We have that $\eps_n \cot(\eps_n X)=\ocal(X^{-1})$ 
and moreover, since $\gamma>1/2$, $n^{-2\gamma}\log n=o\left(1/n\right)=o(\eps_n)$ so we get $n^{-2\gamma}\log n=o(X^{-1})$ for $n$ big enough, which proves \eqref{eq:x}. For \eqref{eq:1/x2+eps2}, using again that $\eps_n \cot(\eps_n X)=\ocal(X^{-1})$ and $n^{-2\gamma}\log n=o(X^{-1})$ we have 
$$x^2+\eps_n^2 = {\eps_n^2}+{\eps_n^2} \cot^2(\eps_n X) + o\left(X^{-2}\right)
= \frac{\eps_n^2}{\sin^2(\eps_n X)}+ o\left(X^{-2}\right)$$
hence 
\begin{align*}
\frac{1}{x^2+{\eps_n^2}} \sim \frac{\sin^2(\eps_n X) }{{\eps_n^2}} = \ocal\left(X^2\right).
\end{align*}
Finally, using \eqref{eq:1/x2+eps2} and the fact that $1/C<|Y|<C$ by definition of $\mathcal{R}_n(C)$ we get
$$y=Y(x^2+\eps_n^2)^{\eta/2}\asymp\frac{\eps_n^\rho}{\sin^\rho(\eps_n X)}=\ocal\left(X^{-\rho}\right),$$
proving \eqref{eq:y}.
\end{proof}

\begin{rem}
Observe that $\rcal_n(C)\subset U$ for $n$ large enough, where $U$ is the domain of definition of the family $(g_{\eps})$. If we take $(x,y)\in \mathcal{R}_n(C)$ and denote $(X,Y):=\Phi_{\eps_n}^\iota(x,y)$, then $|\eps_n X|\ge \re(\eps_n X)\ge \frac{\pi k_n}{10 n}$ by definition of $\rcal_n(C)$, so using  \eqref{eq:x}, \eqref{eq:y} and the fact that $\eps_n\sim \pi/n$ we have that $(x,y)\in U$ for $n$ large enough (depending only on $\rcal_n(C)$ and not on $(x,y)$.)
\end{rem}

\begin{lem}\label{lem:x1eps}
Consider $(x,y) \in \mathcal{R}_n(C)$ and denote $(x_1,y_1) = g_{\eps_n}(x,y)$. Then if $n$ is big enough
$$\log \frac{x_1^2+{\eps_n^2}}{x^2+{\eps_n^2}}=2x + o\left( \frac{1}{n} \right) $$
\end{lem}

\begin{proof}
Since $x_1=x+\left(x^2+\eps_n^2\right)\left(1+\ocal(x,\eps_n)\right)+y\ocal(x,y,\eps_n)$ we have
$$x_1^2=x^2+2x\left(x^2+\eps_n^2\right)+\left(x^2+\eps_n^2\right)\ocal_2(x,\eps_n)+y\ocal_2(x,y,\eps_n)$$
hence
$$\frac{x_1^2+\eps_n^2}{x^2+\eps_n^2}=1+2x+\ocal_2(x,\eps_n)+\frac{y}{x^2+\eps_n^2}\ocal_2(x,y,\eps_n).$$
If $(X,Y):=\Phi_{\eps_n}^\iota(x,y)$, we have by definition of $\mathcal{R}_n(C)$ that $X=\ocal(n)$ and $X^{-1}=\ocal\left(n^{-\gamma}\right)$. Then using \eqref{eq:x} and the fact that $\eps_n\sim \pi/n$ we get $\ocal_2(x,\eps_n)=\ocal\left(n^{-2\gamma}\right)=o\left(n^{-1}\right)$ since $2\gamma>1$. Moreover thanks to \eqref{eq:1/x2+eps2} and \eqref{eq:y} we have that 
$\left(x^2+\eps_n^2\right)^{-1}=\ocal\left(n^2\right)$ and $y=\ocal\left(n^{-\gamma\rho}\right)$, so
$$\frac{y}{x^2+\eps_n^2}\ocal_2(x,y,\eps_n)=\ocal\left(\frac1{n^{\gamma\rho+2\gamma-2}}\right)=o\left(\frac1n\right)$$
since $\gamma\rho>2$. Hence 
$$\frac{x_1^2+{\eps_n^2}}{x^2+{\eps_n^2}}=1+2x+o\left(\frac{1}{n}\right)$$
and therefore
$$\log \frac{x_1^2+{\eps_n^2}}{x^2+{\eps_n^2}} = \log \left(1+2x  + o\left(\frac{1}{n}\right)\right) = 2x + o\left(\frac{1}{n}\right).$$
\end{proof}

\begin{lem}\label{lem:y/x}
If $(x,y) \in \mathcal{R}_n(C)$ then for $n$ big enough
\begin{equation}\label{eq:1/x-ieps}
\frac{1}{x\pm i \eps_n}=\ocal(n)
\end{equation}
and 
\begin{equation}\label{eq:yb/x-ieps}
\frac{y b_{\eps_n}(x,y) }{x\pm i\eps_n}=\frac{bxy }{x\pm i\eps_n}+o\left(\frac{1}{n^2}\right)
\end{equation}
\end{lem}

\begin{proof}
If $(X,Y):=\Phi_{\eps_n}^\iota(x,y)$, using \eqref{eq:x} and the fact that $n^{-2\gamma}\log n = o\left(n^{-1}\right) = o(\eps_n)$ we have
\begin{align*}
\frac{1}{x\pm i \eps_n} = \frac{1}{- \eps_n \cot(\eps_n X) + o(\eps_n)\pm i \eps_n}
\end{align*} 
Since $\frac 1{\eps_n}=\ocal(n)$ and $\frac{1}{\cot({\eps_n} X)\pm i}  = \ocal(1)$, because $\re (\eps_n X) \in (0,\pi)$ and $\im (\eps_nX) = \ocal(1)$ by definition of $\mathcal{R}_n(C)$, we get
\begin{equation*}
\frac{1}{x\pm i \eps_n} \sim -\frac{1}{{\eps_n}} \cdot \frac{1}{\cot({\eps_n} X)\pm i}=\ocal(n),
\end{equation*}
proving \eqref{eq:1/x-ieps}. Then, to get \eqref{eq:yb/x-ieps} we just need to prove that $y b_{\eps_n}(x,y)= bxy+ o\left(n^{-3}\right)$.
Since $b_{\eps_n}(x,y)=bx+\ocal\left(x^2,y,\eps_n\right)$ we have, using \eqref{eq:x}, \eqref{eq:y} and the fact that $X^{-1}=\ocal(n^{-\gamma})$,
$$y b_{\eps_n}(x,y) = bxy+\ocal\left(x^2y,y^2, \eps_n y\right)= bxy+\ocal\left(\frac{1}{n^{\gamma \rho +1}} \right),$$
so $y b_{\eps_n}(x,y)= bxy+ o\left(n^{-3}\right)$ because $\gamma\rho>2$.
\end{proof}

\begin{lem}\label{lem:heps}
Consider $(x,y) \in \mathcal{R}_n(C)$ and denote $(x_1,y_1) = g_{\eps_n}(x,y)$. Then for $n$ big enough
$$\frac{y_1}{y}=1+\eta x+q{\eps_n} + o\left(\frac{1}{n}\right).$$
\end{lem}

\begin{proof}
If $(X,Y):=\Phi_{\eps_n}^\iota(x,y)$, we have by definition of $\mathcal{R}_n(C)$ that $X^{-1}=\ocal\left(n^{-\gamma}\right)$. Then using \eqref{eq:x} and \eqref{eq:y} we have
\begin{align*}
y_1
&= 
y+y\left(\eta x+q{\eps_n}+\ocal(x^2, y, x\eps_n, \eps_n^2)\right)+\ocal\left(x^{m+3}\right)+\eps_n\ocal_{m+1}\left(x,\eps_n\right)\\
&=y\left[1+\eta x+q{\eps_n}+ o\left(\frac{1}{n}\right)+\frac1{y}\,\ocal\left(x^{m+3}\right)+\frac1y\,\eps_n\ocal_{m+1}\left(x,\eps_n\right)\right].
\end{align*}
Moreover, using \eqref{eq:y} 
\begin{equation*}
\dfrac1{y}\asymp \frac{\sin^\rho(\eps_n X)}{\eps_n^\rho}=\ocal\left(X^\rho\right).
\end{equation*}
Then, since $\rho-1<m\le \rho$ we obtain from \eqref{eq:x} that
$$\frac1{y}\ocal\left(x^{m+3}\right)=\ocal\left(\frac1{X^{m+3-\rho}}\right)
=\ocal\left(\frac1{n^{(m+3-\rho)\gamma}}\right)=o\left(\frac1n\right)$$ 
and similarly
$$\frac1{y}\ocal\left(x^{m+1}\eps_n\right)=\ocal\left(\frac1{X^{m+1-\rho}}\,\eps_n\right)=\ocal\left(\frac1{n^{(m+1-\rho)\gamma+1}}\right)=o\left(\frac1n\right).$$
Finally, using the fact that $X=\ocal(n)$ by definition of $\mathcal{R}_n(C)$,
$$\frac1{y}\ocal\left(x^{m-j}\eps_n^{2+j}\right)=\ocal\left(X^{\rho-m+j}\eps_n^{2+j}\right)=\ocal\left(\frac1{n^{2+m-\rho}}\right)=o\left(\frac1n\right)$$
for every $j$ with $0\le j\le m$. Therefore we have 
$$y_1=y\left[1+\eta x+q{\eps_n}+o\left(\frac1n\right)\right],$$
concluding the proof.
\end{proof}

\begin{prop}\label{prop:A-estimate}
We have, for $(x,y) \in \mathcal{R}_n(C)$ and $n$ big enough: 
$$A_{\eps_n}(x,y)=o\left(\frac{1}{n}\right); \quad B_{\eps_n}(x,y)= q\eps_n + o\left(\frac{1}{n}\right).$$
\end{prop}

\begin{proof}
Denoting $(x_1,y_1)=g_{\eps_n}(x,y)$, we have:
\begin{align*}
 A_{\eps_n}(x,y)&=w_{\eps_n}(x_1,y_1)-w_{\eps_n}(x,y)-1\\
 &=\frac{1}{2i{\eps_n}} \log\left( \frac{x_1-i{\eps_n}}{x-i{\eps_n}} \cdot \frac{x+i{\eps_n}}{x_1+i{\eps_n}}\right) -1 + \frac{1-a}{2} \log \frac{x_1^2+\eps_n^2}{x^2+\eps_n^2}
\end{align*}
By Lemma \ref{lem:x1eps}, we have 
$$ \frac{1-a}{2} \log \frac{x_1^2+{\eps_n^2}}{x^2+{\eps_n^2}} = (1-a) x+o\Bigl(\frac{1}{n}\Bigr).$$
On the other hand, we compute:
\begin{align*}
\log \frac{x_1-i\eps_n}{x-i\eps_n} &= \log \frac{x-i\eps_n + (x^2+{\eps_n^2}) a_{\eps_n} + y b_{\eps_n}}{x-i\eps_n}
=\log \left(1+ (x+i\eps_n) a_{\eps_n} + \frac{y b_{\eps_n}}{x-i\eps_n} \right) \\
&=(x+i\eps_n) a_{\eps_n}+\frac{bxy}{x- i\eps_n}   -\frac{1}{2} (x+i\eps_n)^2 a_{\eps_n}^2 +\frac{1}{3} x^3+o\left(\frac{1}{n^2}\right)
\end{align*}
using Lemmas~\ref{lem:x2e2} and \ref{lem:y/x} in the last line (and, as in the previous lemmas, the fact that $X^{-1}=\ocal\left(n^{-\gamma}\right)$ by definition of $\mathcal{R}_n(C)$, where $(X,Y):=\Phi_{\eps_n}^\iota(x,y)$). Similarly, we also have:
\begin{align*}
\log \frac{x_1+i\eps_n}{x+i\eps_n} &= \log \frac{x+i\eps_n + (x^2+{\eps_n^2}) a_{\eps_n} + y b_{\eps_n}}{x+i\eps_n} 
=\log \left(1+ (x-i\eps_n) a_{\eps_n} + \frac{y b_{\eps_n}}{x+i\eps_n} \right) \\
&=(x-i\eps_n) a_{\eps_n} +\frac{bxy}{x+ i\eps_n} -\frac{1}{2} (x-i\eps_n)^2 a_{\eps_n}^2 +\frac{1}{3} x^3+o\left(\frac{1}{n^2}\right)
\end{align*}
so that
\begin{align*}
\frac{1}{2i\eps_n}\log\left( \frac{x_1-i\eps_n}{x-i\eps_n} \cdot \frac{x+i\eps_n}{x_1+i\eps_n}\right) 
&=\frac{1}{2i\eps_n} \left[ 2i\eps_n a_{\eps_n} 
+2i\eps_n\frac{bxy}{x^2+\eps_n^2}-2i\eps_n x a_{\eps_n}^2+ o\left(\frac{1}{n^2}\right) \right] \\
&=a_{\eps_n} +\frac{bxy}{x^2+\eps_n^2} - x a_{\eps_n}^2 + o\left(\frac{1}{n}\right) \\
&=1+ (a-1) x+o\left(\frac{1}{n}\right)
\end{align*} 
using again Lemma~\ref{lem:x2e2} in the last line.

Finally, note that we have 
$$B_{\eps_n}(x,y)=\log\frac{t_{\eps_n}(x_1,y_1)}{t_{\eps_n}(x,y)}=\log \frac{y_1}{y}-\frac{\eta}{2}\log \frac{x_1^2+\eps_n^2}{x^2+\eps_n^2}$$
Using Lemma~\ref{lem:heps} we have
$$\log \frac{y_1}{y}=\log \left(1+\eta x+q {\eps_n} +o\left(\frac{1}{n}\right)
\right)= \eta x+q{\eps_n} +o\left(\frac{1}{n}\right)
$$
and using Lemma~\ref{lem:x1eps} we conclude that $B_{\eps_n}(x,y)= q\eps_n + o\left(n^{-1}\right)$.
\end{proof}

\section{Controlling the orbit}\label{sec:control}

The goal of this section is to provide accurate estimates for the position of the orbit 
$\{g_{\eps_n}^k(x,y): 0 \leq k \leq n-N\}$ of a point $(x,y)$ in the parabolic basin of $g:=g_0$.

The strategy, similar to that of \cite{ABDPR16}, is to split this orbit in three regions: 
\begin{enumerate}
\item A first one where $g_{\eps_n}^k(x,y)$ approaches $(0,0)$ shadowing closely the orbit $g^k(x,y)$, which we call "approaching the eggbeater"; this will occur for $0 \leq k \leq k_n$, where $k_n=\lfloor n^\gamma \rfloor$ as in Section~\ref{sec:approximate}.
\item A second one which we call "in the eggbeater", in which $g_{\eps_n}^k(x,y)$ is close to $(0,0)$ and where the effect of the perturbation is relevant. This will be the case for $k_n \leq k \leq n-k_n$.
\item A third one where $g_{\eps_n}^k(x,y)$ gets away from $(0,0)$ again shadowing the dynamics of $g$, which we call "leaving the eggbeater". This will happen for the last $k_n-N$ iterates.
\end{enumerate}

Let us fix some constant $C>0$. Recall that by Lemma~\ref{lem:petalsforg} and Proposition~\ref{prop:incomfatou}, there exists $r^\iota(C)>0$ such that for all $0<r\le r^\iota(C)$ small enough the incoming petal 
$$P^\iota(r,C):=\left\{(x,y) \in \C^2: |x+r|<r, \left|\frac{y}{(-x)^\eta}\right|<C\right\}$$
has the property that if $(x,y) \in P^\iota(r,C)$, then $g^k(x,y) \in P^\iota(r,C+1)$ for all $k \in \N$ and moreover the incoming Fatou coordinate $\Phi^\iota$ is well-defined on $P^\iota(r,C)$.
Recall that we have the following asymptotic expansion for $\Phi^\iota$ as $\re(-1/x)\to+\infty$ in $P^\iota(r,C)$:
$$\Phi^\iota(x,y) =: (w^\iota(x,y), t^\iota(x,y)) =  \left(-\frac{1}{x} + (1-a) \log(-x) + o(1), \frac{y}{(-x)^\eta} + o(1) \right).$$

The map 
$$\Phi_0^\iota(x,y):=\left(-\frac{1}{x}, \frac{y}{(-x)^\eta}\right)$$ 
maps $P^\iota(r,C)$ biholomorphically to $\H_R \times \D(0,C)$, where $\H_R:=\{X \in \C: \re X >R\}$ and 
$R:=(2r)^{-1}$. 

In a similar way, by Proposition~\ref{prop:outgofatou} there exists $r^o(C)>0$ such that for all $0< r\le r^o(C)$ the outgoing Fatou coordinate $\Phi^o$ is well-defined on 
$$P^o(r,C):=\left\{(x,y)\in\C^2:|x-r|<r,\left|\frac{y}{x^\eta}\right|<C\right\}$$
and satisfies 
$$\Phi^o(x,y)=\left(-\frac{1}{x} + (1-a) \log x + o\left(1\right), \frac{y}{x^\eta} + o\left(1\right)\right),$$
as $\re (-1/x)\to-\infty$ inside $P^o(r,C)$. We set 
$$\Phi_0^o(x,y):=\left(-\frac{1}{x}, \frac{y}{x^\eta}\right),$$  
which maps $P^o(r,C)$ biholomorphically to $-\H_R \times \D(0,C)$, where $-\H_R:=\{X \in \C: \re X <-R\}$.

Throughout this section, we will use the following notations: given $(x_0,y_0)$, we will denote (when defined)
$$(x_j,y_j):=g^j(x_0,y_0), \quad (X_j,Y_j):=\Phi_0^\iota(x_j,y_j), \quad (X_j^o,Y_j^o):=\Phi_0^o(x_j,y_j)$$
$$(x_j^{\eps_n},y_j^{\eps_n}):=g_{\eps_n}^j(x_0,y_0), \quad (X_j^{\eps_n},Y_j^{\eps_n}):=\Phi_0^\iota(x_j^{\eps_n},y_j^{\eps_n}), \quad (X_j^{\eps_n,o},Y_j^{\eps_n,o}):=\Phi_0^o(x_j^{\eps_n},y_j^{\eps_n}).$$ 
We will also denote, for any $t,s\in\C$,
$$A_{t,s}(X,Y):=\left(X+t,e^{\pi s}Y\right).$$

\strut

Given $C>0$, we choose $0<r<2/5$ small enough  (or equivalently $R=(2r)^{-1}>5/4$ large enough) such that
\begin{enumerate}
\item[$(R_1)$]\label{item:choiceRinit} $P^\iota(r,C),P^o(r,Ce^{1+\pi|\re q|}) \subset U$, where $U$ is the domain of definition of $(g_{\eps})$, and there exists a constant $K_1>0$ such that 
$$|X_1-X-1|<1/4,\qquad |Y_1-Y|<K_1\frac1{|X|^2}$$
for all $(X,Y) \in \Phi_0^\iota(P^\iota(r,C))$, where $(X_1,Y_1):=\Phi_0^\iota \circ g \circ \left(\Phi_0^\iota \right)^{-1}(X,Y)$, and
$$|X_1^o-X-1|<1/4, \qquad |Y_1^o-Y|<K_1\frac1{|X|^2}$$
for all $(X,Y)\in \Phi_0^o(P^o(r,Ce^{1+\pi|\re q|}))$, where $(X_1^o,Y_1^o):=\Phi_0^o \circ g \circ \left(\Phi_0^o \right)^{-1}(X,Y)$. 
\end{enumerate}
The existence of such $R$ is guaranteed by the expression \eqref{eq:expressionG} for $\Phi_0^\iota \circ g \circ \left(\Phi_0^\iota \right)^{-1}$ and the fact that $\Phi_0^o \circ g \circ \left(\Phi_0^o \right)^{-1}$ has the same expansion.

We will also later in this section need to possibly increase $R$ further to meet some extra conditions (see Definition \ref{defi:choiceofR}). To ensure the lack of circular definitions, we will explicitly mention that several constants appearing in the following computations do not depend on $R$ (but are allowed to depend on $C$).

\begin{defi}
We let 
$$U_n(R,C):=\{(X, Y)\in \H_R\times \D(0,C): |X|<10 k_n \}  \subsetneq \Phi_0^\iota\left(P^\iota\left((2R)^{-1},C\right)\right).$$ 
\end{defi}

\begin{lem}\label{lem:estimateGeps}
For any $C>0$ and any $R$ satisfying hypothesis $(R_1)$, the map $G_{\eps_n}:=\Phi_0^\iota \circ g_{\eps_n} \circ \left(\Phi_0^\iota\right)^{-1}$ is well-defined on $U_n(R,C)$ for all $n$ large enough. Moreover, there exists a constant $C^\iota>0$ depending on $C$ but not on $R$ (in the sense that it does not increase when we increase $R$) such that if $(X,Y) \in U_n(R,C)$ and $n$ is large enough then 
$$|X_1^{\eps_n} -X_1| \leq C^\iota \left(\frac{1}{n} + \frac{|X|^2}{n^2}  \right) \quad\text{and}\quad
|Y_1^{\eps_n} -Y_1| \leq C^\iota \frac{1}{n}$$
where $(X_1^{\eps_n},Y_1^{\eps_n}):=G_{\eps_n}(X,Y)$ and $(X_1,Y_1):=\Phi_0^\iota \circ g \circ \left( \Phi_0^\iota\right)^{-1}(X,Y)$. 
\end{lem}

\begin{proof}
Let us first prove that $G_{\eps_n}$ is well-defined on $U_n(R,C)$ for $n$ large enough. It suffices to show that if $(X,Y)  \in U_n(R,C)$ and $n$ is large enough then $\re (x_1^{\eps_n})<0$, where $(x_1^{\eps_n},y_1^{\eps_n}):=g_{\eps_n} \circ \left(\Phi_0^\iota\right)^{-1}(X,Y)$. Set $(x,y):=\left(\Phi_0^\iota\right)^{-1}(X,Y)$ and $(x_1,y_1):=g(x,y)$. By the expression \eqref{eq:gsect4} of $g_{\eps_n}$, we have
\begin{equation}\label{eq:x1epsn}
x_1^{\eps_n}=x_1+\ocal\left(x^2\eps_n,y\eps_n,\eps_n^2\right); \qquad y_1^{\eps_n}=y_1+\ocal(y\eps_n)+\eps_n\ocal_{m+1}(x,\eps_n)
\end{equation}
so using the fact that $y=\ocal(x^2)$ because $(x,y)\in P^\iota\left((2R)^{-1},C\right)$ we get
\begin{align*}
-\frac{1}{x_1^{\eps_n}} &= -\frac{1}{x_1 + \ocal\left(x^2 \eps_n, y \eps_n, \eps_n^2\right)}   
=-\frac{1}{x_1 + \ocal\left(\frac{1}{n^2}, \frac{x^2}{n}\right)}
=-\frac{1}{-\frac{1}{X_1} + \ocal\left(\frac{1}{n^2}, \frac{1}{nX^2}\right) }\\
&=\frac{X_1}{1-\ocal\left(\frac{X_1}{n^2}, \frac{X_1}{nX^2}\right)}  =   X_1 + \ocal\left(\frac{X_1^2}{n^2}, \frac{X_1^2}{nX^2}\right).
\end{align*}
Since $|X_1-X-1|<1/4$ by our choice of $R$, we have that $X_1=\ocal(X)$ so we get
$$-\frac{1}{x_1^{\eps_n}}=X_1+\ocal\left(\frac1n,\frac{X^2}{n^2}\right).$$
Note that in all the previous computations, the implicit constants in $\ocal$ depend only on $C$, not on $R$, in the sense that the do not increase if we increase $R$.
By definition of $U_n(R,C)$, we have 
$$\frac{|X|^2}{n^2} \leq  \frac{100 k_n^2}{n^2}.$$
Moreover, since $\re X_1>\re X + 3/4$, there exists a constant $C_1>0$ independent from $R$ such that 
$$\re \left(-\frac{1}{x_1^{\eps_n}} \right) > \re X+ \frac{3}{4} - C_1 \left(\frac{1}{n} + \frac{100}{n^{2-2\gamma}}  \right)$$
so $\re\left(x_1^{\eps_n}\right)<0$ for $n$ large enough, as desired. This computation also gives the first estimate.
	
Let us prove the second estimate. By the mean value inequality,
\begin{align*}
|Y_1^{\eps_n} - Y_1| 
&= \left|\frac{y_1^{\eps_n}}{(-x_1^{\eps_n})^{\eta}} - \frac{y_1}{(-x_1)^\eta}\right| \\
&\leq \sup_{x \in [x_1,x_1^{\eps_n}]} \left|\eta\frac{y_1}{(-x)^{\eta+1}}\right| \cdot |x_1^{\eps_n}-x_1| 
+  \left|\frac1{(-x_1^{\eps_n})^\eta}\right| \cdot |y_1^{\eps_n}-y_1| \\
&\le C_2 \frac{|y_1|}{|x_*|^{\rho+1}} |x_1^{\eps_n}-x_1| 
+  C_2 \frac1{|x_*|^{\rho}}|y_1^{\eps_n}-y_1|
\end{align*}
for some $C_2>0$, where $|x_*|:=\min\{|z|, z \in [x_1, x_1^{\eps_n}]\}$.
By the expression of $g$,
$$x_1=x\left(1+\ocal(x,y)\right)+\ocal(y^2)=x\left(1+\ocal(x)\right),$$
using again the fact that $y=\ocal(x)$ because $(x,y)\in P^\iota\left((2R)^{-1},C\right)$.
Since $R$ was taken large enough so that $|X_1|<|X|+2$, we have that 
we have that 
$$\frac{|x_1|}{|x|}>\frac1{1+4r},$$
where $r=(2R)^{-1}$.  Moreover, using \eqref{eq:x1epsn} we have that $x_1^{\eps_n}=x_1+\ocal\left(x^2\eps_n,\eps_n^2\right)$, so
$$\frac{x_*}{x}=\frac{x_1+\ocal\left(x^2\eps_n,\eps_n^2\right)}{x}=\frac{x_1}{x}+\ocal\left(\frac1n,\frac{X}{n^2}\right).$$
Since $(X,Y)\in U_n(R,C)$, we have that $|X|<10k_n$, so we can take $n$ big enough so that 
\begin{equation}\label{eq:x*/x}
\frac{|x_*|}{|x|}>\frac12\frac 1{1+4r}.
\end{equation}
Hence using the expression of $g$ we have that for $n$ large enough
$$\frac{|y_1|}{|x_*|^{\rho+1}}\le C_3 \frac{|y_1|}{|x|^{\rho+1}}\le C_3\frac1{|x|^{\rho+1}}\left|y(1+\ocal(x,y))+\ocal(x^{m+3})\right|\le C_4 \frac1{|x|}$$ 
for some constants $C_3,C_4>0$, so
$$\frac{|y_1|}{|x_*|^{\rho+1}} |x_1^{\eps_n}-x_1| \le C_4\frac{1}{|x|} |x_1^{\eps_n}-x_1|\le  C_5\left(|x|\frac1n+\frac{|X|}{n^2}\right)\le C_6\frac1n$$
for some constants $C_5,C_6>0$, where we used in the last inequality the fact that $|X|<10k_n$. 
Finally, using \eqref{eq:x1epsn} and \eqref{eq:x*/x} we have
\begin{equation*}\label{eq:y1eps-y1}
\frac{|y_1^{\eps_n}-y_1|}{\left|x_*\right|^\rho} = \frac{\ocal(y \eps_n)+\eps_n\ocal_{m+1}(x,\eps_n)}{|x|^\rho} =\ocal\left(\frac1n\right)
\end{equation*}
using again in the last identity the fact that $|X|\le 10k_n$.
Putting everything together, we obtain
$$Y_1^{\eps_n} - Y_1 = \ocal\left( \frac{1}{n}\right),$$
and the Lemma is proved. \end{proof}

For the outgoing petal, we have the following analogous Lemma:

\begin{defi}
We let 
$$U_n^o(R,Ce^{1+\pi |\re q|}):=\{(X, Y) \in -\H_R\times \D(0, Ce^{1+\pi |\re q|}): |X|<10 k_n \}.$$
Note that $U_n^o(R,Ce^{1+\pi |\re q|}) \subsetneq \Phi_0^o\left(P^o\left((2R)^{-1},Ce^{1+\pi |\re q|}\right)\right).$
\end{defi}

\begin{lem}\label{lem:estimateGepso}
For any $C>0$ and any $R$ satisfying hypothesis $(R_1)$, the map $G_{\eps_n}^o:=\Phi_0^o\circ g_{\eps_n} \circ \left(\Phi_0^o\right)^{-1}$ is well-defined on $U_n^o(R,Ce^{1+\pi |\re q|})$ for all $n$ large enough. Moreover, there exists a constant $C^o>0$ depending on $C$ but not on $R$ (in the sense that it does not increase when we increase $R$) such that if $(X,Y) \in U_n^o(R,Ce^{1+\pi |\re q|})$ and $n$ is large enough then 
$$|X_1^{\eps_n,o}-X_1^o| \leq C^o \left(\frac{1}{n} + \frac{|X|^2}{n^2}  \right) \quad\text{and}\quad
|Y_1^{\eps_n,o} -Y_1^o| \leq C^o \frac{1}{n}$$
where $(X_1^{\eps_n,o},Y_1^{\eps_n,o}):=G_{\eps_n}^o(X,Y)$ and $(X_1^o,Y_1^o):=\Phi_0^o \circ g \circ \left( \Phi_0^o\right)^{-1}(X,Y)$. 
\end{lem}

\begin{proof}
The proof follows essentially the same steps as the previous one. The only 
modification concerns the argument showing that $\re(x_1^{\eps_n})>0$. In this case, we have 
$$\re \left(-\frac{1}{x_1^{\eps_n}} \right) < \re X+ \frac{5}{4} + C_1 \left(\frac{1}{n} + \frac{100}{n^{2-2\gamma}}  \right)$$
for some constant $C_1>0$. Hence, for $n$ sufficiently large, we have $\re(-1/x_1^{\eps_n})<0$, since by assumption $R>5/4$. 
\end{proof}

We will now state explicitly in which sense $R$ must be taken large enough, or equivalently $r=(2R)^{-1}$ small enough. By Lemma~\ref{lem:estimateGeps}, there exists $C^\iota>0$, depending on $C$ but not on $R$, such that for all $(X,Y) \in U_n(R,C)$ and for $n$ large enough
$$|X_1^{\eps_n} -X_1| \leq C^\iota \left(\frac{1}{n} + \frac{|X|^2}{n^2}  \right) \quad\text{and}\quad
|Y_1^{\eps_n} -Y_1| \leq C^\iota \frac{1}{n}$$
where $(X_1^{\eps_n},Y_1^{\eps_n}):=\Phi_0^\iota\circ g_{\eps_n}\circ \left( \Phi_0^\iota\right)^{-1}(X,Y)$ and $(X_1,Y_1):=\Phi_0^\iota \circ g \circ \left( \Phi_0^\iota\right)^{-1}(X,Y)$. 
Since $|X|<10k_n$, we have that for $n$ large enough $|X_1^{\eps_n} -X_1|<1/4.$
Then by our standing assumption $(R_1)$ there exists $M^\iota>0$ (depending on $C$ but not on $R$) such that for all $(X,Y) \in U_n(R,C)$ and for $n$ large enough we have
\begin{equation}\label{eq:diffX1epsX}
|X_1^{\eps_n}-X-1| < \frac{1}{2}; \qquad |Y_1^{\eps_n}-Y|<M^\iota\left(\frac{1}{n} + \frac{1}{|X|^2} \right).
\end{equation}

Moreover, by Lemma \ref{lem:estimateGepso}, similar estimates also hold on the outgoing petal; more precisely, there exists a constant $M^o>0$ (also depending on $C$ but not on $R$) such that for all $(X,Y) \in U_n^o(R,Ce^{1+\pi |\re q|})$ and for $n$ large enough we have 
\begin{equation}\label{eq:diffX1epsXo}
|X_1^{\eps_n,o}-X-1| < \frac{1}{2}; \qquad |Y_1^{\eps_n,o}-Y|< M^o \left(\frac{1}{n} + \frac{1}{|X|^2} \right)
\end{equation}
where $(X_1^{\eps_n,o},Y_1^{\eps_n,o}):=\Phi_0^o\circ g_{\eps_n}\circ \left( \Phi_0^o\right)^{-1}(X,Y)$.
We let $M:=\max(M^\iota, M^o)$.

\medskip

From now on, we fix a compact $K \subset \bcal_{U,v}$, where $\bcal_{U,v}$ is the parabolic basin associated to $v=(1,0)$. Recall that  $\bcal_{U,v}:=\bigcup_{C>0}\bigcup_{n \geq 0} f^{-n}(P^\iota(r, C))$ for any $r>0$. Then there exists a constant $C>2$ (depending only on $K$) such that for any $r>0$ small enough there exists $n_0=n_0(r)\in \N$ such that 
$L:=g^{n_0}(K) \Subset P^\iota(r,C-1)$.

\begin{defi}[Choice of $R$]\label{defi:choiceofR}
From now on, we fix $R>5/2$ (equivalently, $0<r=(2R)^{-1}<1/5$ small enough) satisfying hypothesis $(R_1)$ and large enough such that \\[-7pt]
\begin{enumerate}
\item[$(R_2)$] $r\le r^\iota(C)$ and $r\le r^o(Ce^{1+\pi|\re q|})$. \\[-8pt]
\item[$(R_3)$] $ M \sum_{j=0}^{\infty} (R+ j/2)^{-2} \leq 1/10.$\\[-8pt]
\item[$(R_4)$]  For all $(X,Y) \in \H_{R-1} \times \D(0,C+1)$ 
$$\|\Phi^\iota\circ(\Phi_0^\iota)^{-1}(X,Y)-(X-(1-a)\log X,Y)\|\le \frac1{C^2}$$
and for all $(X,Y) \in -\H_{R-5/2} \times \D(0,Ce^{1+\pi|\re q|}+2)$
$$\|\Phi^o\circ(\Phi_0^o)^{-1}(X,Y)-(X-(1-a)\log(-X),Y)\|\le \frac1{C^2}.$$
\end{enumerate}
\end{defi}

Observe that conditions $(R_2)$ and $(R_4)$ are satisfied for $R$ large enough by Propositions~\ref{prop:incomfatou} and \ref{prop:outgofatou}.

\begin{defi}
Let $\pi_1: \C^2 \to \C$ denote the projection on the first coordinate. 
We define
$$M_C:=\displaystyle \left\{(x,y) \in \C^2:  x \in \pi_1(L) \text{ and }  \frac{1}{C-1} \leq \left|\frac{y}{(-x)^\eta}\right| \leq C-1 \right \}$$
and 
$$N_C:=\displaystyle \left\{(x,y) \in \C^2:  x \in \pi_1(L) \text{ and } \left|\frac{y}{(-x)^\eta}\right| \leq C-1 \right \}.$$
\end{defi}
Observe that, by definition, $M_C\subset N_C$ and $L\subset N_C\subset P^\iota(r,C-1)$.

\subsection{Approaching the eggbeater}\label{sec:incoming}

\begin{lem}\label{lem:takeRbig}
For all $n$ large enough and for all $0 \leq j \leq k_n$ we have that
$$\Phi_0^\iota \circ g_{\eps_n}^j\left(M_C\right) \subset U_n(R,C).$$
In particular, $g_{\eps_n}^j\left(M_C\right) \subset  P^\iota\left(r,C\right)$ for all $n$ large enough and all $0 \leq j \leq k_n$. 
\end{lem}

\begin{proof}
Since $M_C\subset P^\iota(r,C-1)$ is compact, we can assume that $n$ is large enough such that $\Phi_0^\iota(M_C)\subset U_n(R,C)$. Take $(x_0,y_0)\in M_C$, set $(X_0,Y_0):=\Phi_0^\iota(x_0,y_0)$ and denote, for all $j\in\N$ such that it is well-defined, 
$(X_j^{\eps_n},Y_j^{\eps_n}):=\Phi_0^\iota \circ g_{\eps_n} \circ \left(\Phi_0^\iota\right)^{-1}(X_0,Y_0)$. Set 
$$K_n:=\min \{j \in \N : (X_j^{\eps_n},Y_j^{\eps_n}) \notin U_n(R,C)\}>0$$ 
and let us prove that $K_n > k_n$. 
	
First, recall that by \eqref{eq:diffX1epsX} there exists a constant $M>0$ such that for all $j \leq K_n$ and for $n$ large enough we have 
$$|X_{j}^{\eps_n}-X_{j-1}^{\eps_n}-1| < \frac{1}{2}$$
$$|Y_{j}^{\eps_n}-Y_{j-1}^{\eps_n}|< M \left(\frac{1}{n}  + \frac{1}{|X_{j-1}^{\eps_n}|^2} \right).$$
From those inequalities, we get:
\begin{equation}\label{eq:contrX}
|X_{j}^{\eps_n}-X_0-j| \leq \frac{j}{2}  
\end{equation}
and 
\begin{equation}\label{eq:contrY}
|Y_{j}^{\eps_n} - Y_0|\leq \frac{M j}{n} + M\sum_{k=0}^{j-1} \frac{1}{|X_k^{\eps_n}|^{2}}
\end{equation} 
for all $0 \leq j \leq K_n$. 
Using \eqref{eq:contrX} in \eqref{eq:contrY} we get, for all $0 \leq j \leq K_n$, 
\begin{align*}
|Y_{j}^{\eps_n} - Y_0|\leq  \frac{M j}{n} + M\sum_{k=0}^{j-1}\frac{1}{(\re X_0+ k/2)^2}\leq \frac{Mj}{n}+ M \sum_{k=0}^{j-1} \frac{1}{(R+ k/2)^2}. 
\end{align*}
Recall that by condition $(R_3)$ we have that $ M \sum_{k=0}^{\infty} (R+ k/2)^{-2} \leq 1/10$: from now on, we also assume that $n$ is large enough so that $ \frac{M k_n}{n}\leq 1/10.$
Suppose by contradiction that $K_n \leq k_n$. By definition $(X_{K_n}^{\eps_n}, Y_{K_n}^{\eps_n}) \notin U_n(R,C)$; but on the other hand, we have 
$$|X_{K_n}^{\eps_n}- X_0 - K_n| \leq \frac{K_n}{2} \quad \text{ and } \quad |Y_{K_n}^{\eps_n} - Y_{0} | \leq \frac{2}{10}. $$
Therefore: $|Y_{K_n}^{\eps_n}| < C-1+2/10< C$, $R + K_n/2 \leq \re \left(X_{K_n}^{\eps_n}\right)$, and
$$|X_{K_n}^{\eps_n}| \leq |X_0| + \frac{3}{2} K_n \le \max \left\{ \left| - \frac{1}{x} \right| : (x,y) \in M_C  \right\}+ \frac{3}{2} k_n   < 10 k_n $$
(again, up to taking $n$ large enough such that $ \max \left\{ \left| - 1/x \right| : (x,y) \in M_C  \right\}+ 3 k_n/2   < 10 k_n $).
Therefore $(X_{K_n}^{\eps_n}, Y_{K_n}^{\eps_n}) \in U_n(R,C)$, a contradiction. This proves that $K_n>k_n$, hence that 
$(X_k^{\eps_n}, Y_k^{\eps_n}) \in U_n(R,C)$ for all $0 \leq k \leq k_n$.
\end{proof}

\begin{lem}\label{lem:firstkn}
For $n$ large enough 
$$\Phi^\iota\circ g_{\eps_n}^{k_n}\bigr|_{M_C}=A_{k_n,0}\circ\Phi^\iota\bigr|_{M_C}+o(1),$$
where the convergence of the term $o(1)$ is uniform on $M_C$ (recall that $A_{t,s}(X,Y)=(X+t,e^{\pi s}Y)$ for any $t,s\in\C$).
\end{lem}

\begin{proof} First note that, by Lemma~\ref{lem:takeRbig}, we have $\Phi_0^\iota\circ g_{\eps_n}^j\left(M_C\right) \subset U_n(R,C)$ for $n$ large enough and for all $0 \leq j \leq k_n$. In particular, for $n$ large enough and for all $0\le j\le k_n$ we have that $g_{\eps_n}^j\left(M_C\right) \subset  P^\iota\left((2R)^{-1},C\right)$, so $\Phi^\iota \circ g_{\eps_n}^{j}(x_0,y_0)$ is well-defined for any $(x_0,y_0)\in M_C$.
	
Take $(x_0,y_0)\in M_C$ and denote $(X_0,Y_0):=\Phi_0^\iota(x_0,y_0)$, $(X_1,Y_1):=\Phi_0^\iota\circ g(x,y)\in\H_\R\times\D(0,C)$ and $(X_1^{\eps_n},Y_1^{\eps_n}):=\Phi_0^\iota\circ g_{\eps_n}(x_0,y_0)\in\H_R\times\D(0,C)$. Set $\Phi_G:=\Phi^\iota\circ(\Phi_0^\iota)^{-1}: \mathbb{H}_{R-1} \times \D(0,C+1) \to \C^2$, so 
$$\Phi_G(X_1,Y_1)=\Phi_G(X_0,Y_0)+(1,0)$$
and hence
$$\Phi^\iota\circ g_{\eps_n}(x_0,y_0)-\Phi^\iota(x_0,y_0)-(1,0)=\Phi_G(X_1^{\eps_n},Y_1^{\eps_n})-\Phi_G(X_1,Y_1).$$
By Proposition~\ref{prop:incomfatou} and condition $(R_4)$, there exists a holomorphic map $u:  \H_{R-1} \times \D(0,C+1) \to \C^2$ such that 
$$\Phi_G(X,Y)= (X-(1-a)\log X , Y) + u(X,Y)$$
with $\|u(X,Y)\| \leq 1$ for all $(X,Y) \in \H_{R-1} \times \D(0,C+1)$. By Cauchy estimates, we have  
$$\| \partial_X u(X,Y)\|\leq 1  \quad \text{and} \quad \| \partial_Y u(X,Y)\|\leq 1 $$
for all $(X,Y) \in \H_R\times \D(0,C)$ and hence, by the mean value inequality, 
$$\|u(X_1^{\eps_n},Y_1^{\eps_n}) - u(X_1,Y_1)\| \leq |X_1^{\eps_n}-X_1|+|Y_1^{\eps_n}-Y_1|.$$
On the other hand, also by the mean value inequality, we have 
$$|\log X_1^{\eps_n}-\log X_1| < |X_1^{\eps_n}- X_1|. $$
Therefore
\begin{align*}
\|\Phi_G(X_1^{\eps_n},Y_1^{\eps_n})-\Phi_G(X_1,Y_1)\| &\leq |X_1^{\eps_n}-X_1|+|1-a|| \log X_1^{\eps_n} - \log X_1|+|Y_1^{\eps_n}-Y_1| \\
&\qquad\quad + \| u(X_1^{\eps_n},Y_1^{\eps_n}) - u(X_1,Y_1)\| \\
 &\leq (2+|1-a|)\left(|X_1^{\eps_n}-X_1|+|Y_1^{\eps_n}-Y_1|\right).
\end{align*}
Then using Lemma~\ref{lem:estimateGeps} we obtain that for $n$ large enough
$$\|\Phi^\iota\circ g_{\eps_n}(x_0,y_0)-\Phi^\iota(x_0,y_0)-(1,0)\|\le K_2\left(\frac1n+\frac{1}{|x_0|^2n^2}\right),$$
where $K_2= 2(2+|1-a|) C^\iota$. By an immediate induction, we have 
$$	\| \Phi^\iota \circ g_{\eps_n}^{k_n}(x_0,y_0) - \Phi^\iota(x_0,y_0) - (k_n,0) \|
\leq \frac{K_2 k_n}{n}+\frac{K_2}{n^2} \sum_{j=0}^{k_n-1} \frac{1}{|x_j^{\eps_n}|^2},   $$
where $(x_j^{\eps_n},y_j^{\eps_n}):=g_{\eps_n}^j(x_0,y_0)$.
By definition of $U_n(R,C)$ we have $|x_j^{\eps_n}|^{-1} < 10 k_n$ for all $0 \leq j \leq k_n$; therefore
$$	\| \Phi^\iota \circ g_{\eps_n}^{k_n}(x_0,y_0) - \Phi^\iota(x_0,y_0) - (k_n,0) \|
\leq \frac{K_2 k_n}{n}+\frac{100 K_2 k_n^3}{n^2}$$
and using the fact that $k_n/n=o(1)$ and $k_n^3/n^2=o(1)$ the Lemma follows. 
\end{proof}

Recall from Section~\ref{sec:approximate} that 
$$\Phi_{\eps_n}^\iota(x,y) = (w_{\eps_n}(x), t_{\eps_n}(x,y))$$
where
$$w_{\eps_n}(x)=\frac{1}{{\eps_n}} \arctan\left(\frac{x}{{\eps_n}}\right)+\frac\pi{2{\eps_n}} + \frac{1- a}{2} \log(x^2+{\eps_n^2}); \quad t_{\eps_n}(x,y)=\frac{y}{(x^2+\eps_n^2)^{\eta/2}}.$$

\begin{lem}\label{lem:apprfatou}
For $n$ large enough we have 
$$\Phi_{\eps_n}^\iota\circ g_{\eps_n}^{k_n}\bigr|_{M_C}=\Phi^\iota\circ g_{\eps_n}^{k_n}\bigr|_{M_C}+o(1),$$
where the convergence of the term $o(1)$ is uniform on $M_C$.
\end{lem}

\begin{proof}
Take $(x_0,y_0)\in M_C$ and let $(x_{k_n}^{\eps_n}, y_{k_n}^{\eps_n}):=g_{\eps_n}^{k_n}(x_0,y_0)$. By Lemmas~\ref{lem:firstkn}, \ref{lem:takeRbig} and \ref{lem:asympphi-1},
$$x_{k_n}^{\eps_n} = - \frac{1}{k_n}+  \ocal\left(\frac{\log n}{k_n^2}\right)\quad \text{ and }\quad y_{k_n}^{\eps_n} = t^\iota(x_0,y_0) \frac{1}{k_n^\eta} + o\left(\frac{1}{k_n^\rho}\right),$$
where the terms $\ocal(k_n^{-2}\log n)$ and $o(k_n^{\rho})$ are uniform on $M_C$. In particular $x_{k_n}^{\eps_n}\sim -1/k_n$ uniformly on $M_C$ and
$$\frac{\eps_n}{x_{k_n}^{{\eps_n}}} =  \frac{\eps_n}{- \frac{1}{k_n}+  \ocal\left(\frac{\log n}{{k_n^2}}\right) } = -{\eps_n k_n + \eps_n\ocal\left(\log n\right)} = o(1).$$
Since $x_{k_n}^{\eps_n}\sim -1/k_n$ uniformly on $M_C$, for $n$ large enough depending only on $M_C$ we have that $\re(x_{k_n}^{\eps_n}/\eps_n)<0$. Then, using the relation $\arctan z + \arctan\left(1/z\right) = -\pi/2$ whenever $\re z <0$ we have that
\begin{align*}
w_{\eps_n}(x_{k_n}^{\eps_n})&=\frac1{\eps_n}\arctan\left(\frac{x_{k_n}^{\eps_n}}{\eps_n}\right)+\frac{\pi}{2\eps_n}+\frac{1-a}2\log((x_{k_n}^{\eps_n})^2+\eps_n^2)\\
&=- \frac{1}{\eps_n} \arctan \frac{\eps_n}{x_{k_n}^{\eps_n}} + \frac{1 - a}{2} \log \left((x_{k_n}^{\eps_n})^2 + {\eps_n^2}\right)
\end{align*}
and therefore
\begin{align*}
w_{\eps_n}(x_{k_n}^{\eps_n}) &= - \frac{1}{\eps_n} \arctan \frac{\eps_n}{x_{k_n}^{\eps_n}} + \frac{1 - a}{2} \log \left((x_{k_n}^{\eps_n})^2 + {\eps_n^2}\right)  \\
&= - \frac{1}{\eps_n} \left(\frac{\eps_n}{x_{k_n}^{\eps_n}} +{\ocal(\eps_n^3 k_n^3)} \right)+ (1 - a) \log \left(-x_{k_n}^{\eps_n}\right)  + \frac{1-a}{2} \log \left( 1 + \frac{{\eps_n^2}}{k_n^2} \right)+ o(1) \\
&= - \frac{1}{x_{k_n}^{\eps_n}} +  (1 - a) \log \left(-x_{k_n}^{\eps_n}\right)  + o(1)  = w^\iota (x_{k_n}^{\eps_n}, y_{k_n}^{\eps_n}) + o(1).
\end{align*}
Similarly:
\begin{align*}
t_{\eps_n}(x_{k_n}^{\eps_n},y_{k_n}^{\eps_n}) &= \frac{y_{k_n}^{\eps_n}}{\left((x_{k_n}^{\eps_n})^2+ {\eps_n^2}\right)^{{\eta}/2}} = \frac{y_{k_n}^{\eps_n}}{(-x_{k_n}^{\eps_n})^\eta} \left(1+ \frac{\eps_n^2}{(x_{k_n}^{\eps_n})^2}\right)^{-{\eta}/2} \\
&= \left(t^\iota(x_{k_n}^{\eps_n},y_{k_n}^{\eps_n})+o(1)\right) (1 + o(1)) =  t^\iota(x_{k_n}^{\eps_n},y_{k_n}^{\eps_n}) + o(1)
\end{align*}
(in the last line, we used the fact that $t^\iota(x_{k_n}^{\eps_n}, y_{k_n}^{\eps_n}) = \ocal(1)$ by Lemma~\ref{lem:firstkn}).
\end{proof}

\subsection{In the eggbeater}

\begin{defi}
Let $\hat C:=C e^{1/2+\pi |\re q|}>C$ and $\tilde C:=C e^{1+\pi |\re q|}>C$ (recall that $C$ was fixed when we fixed the compact $K\subset\mathcal{B}_{U,v}$).
\end{defi}

\begin{lem}\label{lem:entering}
For all $n$ large enough,  we have $g_{\eps_n}^{k_n}(M_C) \subset \rcal_n(C)$ (recall that $\rcal_n(C)$ is the set from Definition~\ref{defi:Rn} and that for $n$ large enough $\rcal_n(C)\subset U$, where $U$ is the domain of definition of $(g_{\eps_n})$).
\end{lem}

\begin{proof}
Let $(x_0,y_0) \in M_C$ and  $(x_{k_n}^{\eps_n}, y_{k_n}^{\eps_n}):=g_{\eps_n}^{k_n}(x_0,y_0)$. By Lemmas~\ref{lem:firstkn} and \ref{lem:apprfatou}, we have
$$\Phi_{\eps_n}^\iota\circ g_{\eps_n}^{k_n}(x_0,y_0)   = \Phi^\iota(x_0,y_0) +(k_n,0)+ o(1),$$
where the convergence of the term $o(1)$ is uniform on $M_C$. Therefore
$$t_{\eps_n}(x_{k_n}^{\eps_n},y_{k_n}^{\eps_n}) = t^\iota(x_0,y_0) + o(1).$$
Since $(C-1)^{-1}\le |y_0(-x_0)^{-\eta}|\le C-1$ by definition of $M_C$ and $|t^\iota(x_0,y_0)-y_0(-x_0)^{-\eta}|<1/C^{2}<1/4$ by condition $(R_4)$, we have that $(C-1)^{-1}-C^{-2}\le |t^\iota(x_0,y_0)|\le C-1/4$ and then $C^{-1}<|t_{\eps_n}(x_{k_n}^{\eps_n}, y_{k_n}^{\eps_n})| < C$ for all $n$ large enough (depending only on $M_C$ and not on the choice of $(x_0,y_0)$). Similarly
\begin{align*}
w_{\eps_n}(x_{k_n}^{\eps_n}) &= w^\iota(x_0,y_0) + k_n + o(1) 
\end{align*}
and therefore, since $\eps_n=\pi/n+o(1/n)$,
$$\eps_n w_{\eps_n}(x_{k_n}^{\eps_n}) =\eps_n k_n + \ocal(\eps_n) = \frac{\pi k_n}{n} + \ocal\left(\eps_n \right). $$
where the implicit constants in the $\ocal$ terms only depend on $M_C$ and not on the choice of $(x_0,y_0)$.
In particular, for all $n$ large enough (depending only on $M_C$)
$$\re(\eps_n w_{\eps_n}(x_{k_n}^{\eps_n})) \in \left[\frac{\pi k_n}{10 n}, \pi - \frac{\pi k_n}{10n}\right] \quad  \text{and} \quad |\im (\eps_n w_{\eps_n}(x_{k_n}^{\eps_n})) | \leq C \frac{\pi}n,$$
hence $(x_{k_n}^{\eps_n}, y_{k_n}^{\eps_n}) \in \rcal_n(C)$ for $n$ large enough.
\end{proof}

Recall that $\Phi_{\eps_n}^o:=\Phi_{\eps_n}^\iota - \left(\pi/\eps_n,0 \right)$.

\begin{prop}\label{prop:passing}
For $n$ large enough we have that $g_{\eps_n}^{n-k_n}(M_C)\subset \mathcal{R}_n(\hat C)$  and
$$\Phi_{\eps_n}^o\circ g_{\eps_n}^{n-k_n}\bigr|_{M_C}=A_{\sigma-2k_n,q}\circ\Phi_{\eps_n}^\iota\circ g_{\eps_n}^{k_n}\bigr|_{M_C}+o(1),$$
where the convergence of the term $o(1)$ is uniform on $M_C$.
\end{prop}

\begin{proof}
Take $(x_0,y_0)\in M_C$ and let $(x_j^{\eps_n},y_j^{\eps_n}):=g_{\eps_n}^j(x_0,y_0)$. Let us start by proving that for $n$ large enough (depending only on $M_C$) and for $0\le j\le n- 2k_n$ we have 
$(x_{k_n+j}^{\eps_n},y_{k_n+j}^{\eps_n}) \in \rcal_n(\hat C)$  and 
\begin{equation*}
w_{\eps_n} (x_{k_n+j}^{\eps_n}) =  w_{\eps_n} (x_{k_n}^{\eps_n})+ j +\sum_{k=0}^{j-1} A_{\eps_n}(x_{k_n+k}^{\eps_n},y_{k_n+k}^{\eps_n})
\end{equation*}
and 
\begin{equation*}
t_{\eps_n}(x_{k_n+j}^{\eps_n}, y_{k_n+j}^{\eps_n}) =  t_{\eps_n}(x_{k_n}^{\eps_n}, y_{k_n}^{\eps_n})
\exp\left( \sum_{k=0}^{j-1} B_{\eps_n}(x_{k_n+k}^{\eps_n},y_{k_n+k}^{\eps_n}) \right)
\end{equation*}
where, as in subsection~\ref{subsec:estimationerror},
$$A_{\eps_n}(x,y):=w_{\eps_n}(x_1)- w_{\eps_n}(x) -1 \quad \text{ and }\quad B_{\eps_n}(x,y):=\log\frac{t_{\eps_n}(x_1,y_1)}{t_{\eps_n}(x,y)},$$ 
with $(x_1,y_1):=g_{\eps_n}(x,y)$. 

We argue by induction on $j$. Indeed, the assertions hold for $j=0$ by Lemma \ref{lem:entering} and the fact that $\rcal_n(C) \subset \rcal_n(\hat C)$. Moreover, if they hold for some $0\le j\le n-2k_n-1$, then thanks to Proposition~\ref{prop:A-estimate} (applied on $\rcal_n(\hat C)$), we have
\begin{align*}
w_{\eps_n} (x_{k_n+j+1}^{\eps_n}) 
&= w_{\eps_n}(x_{k_n+j}^{\eps_n})+1+ A_{\eps_n}(x_{k_n+j}^{\eps_n},y_{k_n+j}^{\eps_n})\\
&=w_{\eps_n} (x_{k_n}^{\eps_n})+ j+1 + \sum_{k=0}^{j} A_{\eps_n}(x_{k_n+k}^{\eps_n},y_{k_n+k}^{\eps_n})\\
&=w_{\eps_n} (x_{k_n}^{\eps_n})+ j+1 + o(1) \\
&=k_n + j +1 + \ocal(1),
\end{align*}
where the terms $o(1)$ and $\ocal(1)$ are uniform, and using in the last line that $w_{\eps_n}(x_{k_n}^{\eps_n})= w^\iota(x_0,y_0) + k_n + o(1)$ by Lemmas \ref{lem:firstkn} and \ref{lem:apprfatou}. Moreover,
\begin{align*}
t_{\eps_n} (x_{k_n+j+1}^{\eps_n}, y_{k_n+j+1}^{\eps_n}) 
&= t_{\eps_n} (x_{k_n+j}^{\eps_n}, y_{k_n+j}^{\eps_n}) e^{B_{\eps_n}(x_{k_n+j}^{\eps_n},y_{k_n+j}^{\eps_n})}\\
&=t_{\eps_n} (x_{k_n}^{\eps_n}, y_{k_n}^{\eps_n})\exp\left(\sum_{k=0}^{j}  B_{\eps_n}(x_{k_n+k}^{\eps_n},y_{k_n+k}^{\eps_n})\right)\\		
&=t_{\eps_n} (x_{k_n}^{\eps_n}, y_{k_n}^{\eps_n})e^{qj\eps_n + o(1)},
\end{align*}
where the term $o(1)$ is uniform on $M_C$. Therefore, using the fact that $n\left(\eps_n-\pi/n\right)=\ocal(\eps_n)$ we have, for $n$ big enough depending only on $M_C$, $\eps_n w_{\eps_n}(x_{k_n+j+1}^{\eps_n}) = \frac{\pi}{n} (k_n+j+1+\ocal(1))$, 
hence
\[ \frac{\pi k_n}{10n}\le\re(\eps_n w_{\eps_n}(x_{k_n+j+1}^{\eps_n})) \leq \pi-\frac{\pi k_n}{10n}. \]
And since $\eps_n w_{\eps_n}(x_{k_n+j+1}^{\eps_n})=\eps_n\left[w_{\eps_n} (x_{k_n}^{\eps_n})+ j+1 + o(1)\right]$ and $|\im(\eps_n w_{\eps_n} (x_{k_n}^{\eps_n}))|\le C\pi/n$ by Lemma~\ref{lem:entering}, for $n$ large enough depending only on $M_C$ we get, using the fact that $\eps_n=\pi/n+o(1)$,
\[ |\im(\eps_n w_{\eps_n} (x_{k_n+j+1}^{\eps_n}))|< \hat C\frac{\pi}{n}. \]
Moreover, since
\[ |t_{\eps_n} (x_{k_n+j+1}^{\eps_n}, y_{k_n+j+1}^{\eps_n})| = |t_{\eps_n} (x_{k_n}^{\eps_n}, y_{k_n}^{\eps_n})e^{qj\eps_n + o(1)}|, \]
and 
\[\frac{1}{C} < | t_{\eps_n}(x_{k_n}^{\eps_n}, y_{k_n}^{\eps_n}) | <  C
\]
for $n$ large enough, by Lemma~\ref{lem:entering}, we get that
\[ \frac{1}{\hat C} < | t_{\eps_n}(x_{k_n+j+1}^{\eps_n}, y_{k_n+j+1}^{\eps_n}) | < \hat C \]
for $n$ large enough depending only on $M_C$, and the statement is proved. 

Taking $j= n-2k_n$ we obtain that $(x_{n-k_n}^{\eps_n},y_{n-k_n}^{\eps_n})\in\mathcal{R}_n(\hat C)$ for $n$ large enough and
\begin{equation*}
w_{\eps_n} (x_{n-k_n}^{\eps_n}) =  w_{\eps_n} (x_{k_n}^{\eps_n})+ n-2k_n +\sum_{k=0}^{n-2k_n-1} A_{\eps_n}(x_{k_n+k}^{\eps_n},y_{k_n+k}^{\eps_n})
\end{equation*}
and 
\begin{equation*}
t_{\eps_n}(x_{n-k_n}^{\eps_n}, y_{n-k_n}^{\eps_n}) =  t_{\eps_n}(x_{k_n}^{\eps_n}, y_{k_n}^{\eps_n})
\exp\left( \sum_{k=0}^{n-2k_n-1} B_{\eps_n}(x_{k_n+k}^{\eps_n},y_{k_n+k}^{\eps_n}) \right)
\end{equation*}
so using again Proposition~\ref{prop:A-estimate} and the fact that $\eps_n=\pi/n+o(1/n)$ we get
\[\left(w_{\eps_n} (x_{n-k_n}^{\eps_n}), t_{\eps_n} (x_{n-k_n}^{\eps_n}, y_{n-k_n}^{\eps_n})\right) = \left(w_{\eps_n} (x_{k_n}^{\eps_n})+ n -2k_n+ o(1) , e^{\pi q + o(1)} t_{\eps_n} (x_{k_n}^{\eps_n}, y_{k_n}^{\eps_n})\right),	\]	
where the terms $o(1)$ are uniform on $M_C$. Finally, since $\Phi_{\eps_n}^o(x,y) = \Phi_{\eps_n}^\iota(x,y) - (\pi/\eps_n, 0)$ we obtain, using the fact that $n -\pi/\eps_n  = \sigma+o(1)$, 
$$\Phi_{\eps_n}^o \circ g_{\eps_n}^{n-k_n}(x_0,y_0) = \left(w_{\eps_n} (x_{k_n}^{\eps_n})+ \sigma -2k_n, e^{\pi q} t_{\eps_n} (x_{k_n}^{\eps_n}, y_{k_n}^{\eps_n})\right)+ o(1),$$
where the term $o(1)$ is uniform on $M_C$.
\end{proof}

\subsection{After the eggbeater}

We will now estimate the orbit of $g_{\eps_n}$ as it gets away from the origin in the outgoing petal. Recall that $\tilde C:=Ce^{1+\pi|\re q|}$ and
$$U_n^o(R,\tilde C):=\left\{(X,Y) \in -\H_R\times \D(0, \tilde C): |X|<10 k_n \right\} .$$

\begin{lem}[Compare to Lemma~\ref{lem:apprfatou}]\label{prop:afteregg}
For $n$ large enough $\Phi_0^o\circ g_{\eps_n}^{n-k_n}(M_C)\subset U_n^o(R,\tilde C)$ (so in particular $g_{\eps_n}^{n-k_n}(M_C)\subset  P^o(r,\tilde C)$) and
$$\Phi_{\eps_n}^o\circ g_{\eps_n}^{n-k_n}\bigr|_{M_C}=\Phi^o\circ g_{\eps_n}^{n-k_n}\bigr|_{M_C}+o(1),$$
where the convergence of the term $o(1)$ is uniform on $M_C$.
\end{lem}

\begin{proof}
Take $(x_0,y_0)\in M_C$ and let $(x_{n-k_n}^{\eps_n},y_{n-k_n}^{\eps_n}):=g_{\eps_n}^{n-k_n}(x_0,y_0)$. By Proposition~\ref{prop:passing}, we have
that $(x_{n-k_n}^{\eps_n},y_{n-k_n}^{\eps_n})\in\mathcal{R}_n(\hat C)$ and
\begin{align*}
w_{\eps_n}(x_{n-k_n}^{\eps_n})-  \frac{\pi}{\eps_n} &= w_{\eps_n}(x_{k_n}^{\eps_n})+\sigma-2k_n+o(1)
\end{align*}
so using Lemmas \ref{lem:firstkn} and \ref{lem:apprfatou} we get
\begin{align*}
w_{\eps_n}(x_{n-k_n}^{\eps_n})-  \frac{\pi}{\eps_n} &=w^\iota(x_0,y_0)+\sigma- k_n + o(1)=-k_n+\ocal(1)
\end{align*}
for $n$ large enough, where the term $\ocal(1)$ is uniform on $M_C$. By Lemma~\ref{lem:x2e2} (applied on $\mathcal{R}_n(\hat C)$) we have
$$x_{n-k_n}^{\eps_n} = -\eps_n \cot(\eps_n 	w_{\eps_n}(x_{n-k_n}^{\eps_n}) ) + \ocal\left(\frac{\log n}{n^{2\gamma}}\right).$$
Therefore for $n$ large enough
$$x_{n-k_n}^{\eps_n}= -\eps_n \cot\left(\pi-\eps_n k_n +\ocal(\eps_n)\right)+ \ocal\left(\frac{\log n}{n^{2\gamma}}\right)\sim \frac{1}{k_n}$$
uniformly on $M_C$. Therefore, for $n$ large enough depending only on $M_C$ we have that 
$$\re\left(\frac{-1}{x_{n-k_n}^{\eps_n}}\right)<-R, \quad\text{and}\quad \left|\frac1{x_{n-k_n}^{\eps_n}}\right|<10k_n.$$ 
Moreover, since 
\begin{align*}
t_{\eps_n}(x_{n-k_n}^{\eps_n},y_{n-k_n}^{\eps_n}) = \frac{y_{n-k_n}^{\eps_n}}{((x_{n-k_n}^{\eps_n})^2+{\eps_n^2})^{\eta/2}} =\frac{y_{n-k_n}^{\eps_n}}{(x_{n-k_n}^{\eps_n})^\eta} \left(1+\frac{\eps_n^2}{(x_{n-k_n}^{\eps_n})^2}\right)^{-\frac{\eta}2} 
\end{align*}
we obtain, using the fact that $x_{n-k_n}^{\eps_n}\sim 1/k_n$ and the definition of $\mathcal{R}_n(\hat C)$, that 
$$\left|\frac{y_{n-k_n}^{\eps_n}}{(x_{n-k_n}^{\eps_n})^{\eta}}\right|<\hat C e^{1/2}=\tilde C$$ 
for $n$ large enough depending only on $M_C$, so $\Phi_0^o(x_{n-k_n}^{\eps_n},y_{n-k_n}^{\eps_n})\in U_n^o(R,\tilde C)$ and in particular $(x_{n-k_n}^{\eps_n},y_{n-k_n}^{\eps_n})\in P^o(r,\tilde C)$ for $n$ large enough.

Now, if we denote $(w^o(x,y), t^o(x,y)):=\Phi^o(x,y)$, we want to prove that
\begin{equation*}
w_{\eps_n}(x_{n-k_n}^{\eps_n})- \frac{\pi}{\eps_n} = w^o(x_{n-k_n}^{\eps_n}, y_{n-k_n}^{\eps_n}) + o(1)
\end{equation*}
and
\begin{equation*}
t_{\eps_n}(x_{n-k_n}^{\eps_n}, y_{n-k_n}^{\eps_n}) = t^o(x_{n-k_n}^{\eps_n}, y_{n-k_n}^{\eps_n}) + o(1).
\end{equation*} 
Since $x_{n-k_n}^{\eps_n}\sim 1/k_n$ uniformly on $M_C$, for $n$ large enough depending only on $M_C$ we have that $\re(x_{n-k_n}^{\eps_n}/\eps_n)>0$. Then, using the relation $\arctan z + \arctan\left(1/z\right) = \pi/2$ whenever $\re z >0$ and the definition of $w_{\eps_n}$ we have that for $n$ large enough
\begin{align*}
w_{\eps_n}(x_{n-k_n}^{\eps_n})- \frac{\pi}{\eps_n}&=\frac{1}{\eps_n} \arctan\left(\frac{x_{n-k_n}^{\eps_n}}{\eps_n}\right) - \frac{\pi}{2\eps_n} + \frac{1-a}{2} \log((x_{n-k_n}^{\eps_n})^2+{\eps_n^2}) \\
&=- \frac{1}{\eps_n} \arctan \left(\frac{\eps_n}{x_{n-k_n}^{\eps_n}}\right) + \frac{1-a}{2} \log((x_{n-k_n}^{\eps_n})^2+{\eps_n^2}).
\end{align*}
Since $\frac{\eps_n}{x_{n-k_n}^{\eps_n}} = o(1)$ because $x_{n-k_n}^{\eps_n}\sim 1/k_n$, we get
\begin{align*}
w_{\eps_n}(x_{n-k_n}^{\eps_n}) - \frac{\pi}{\eps_n} &= - \frac{1}{\eps_n} \arctan \left(\frac{\eps_n}{x_{n-k_n}^{\eps_n}}\right) + \frac{1-a}{2} \log((x_{n-k_n}^{\eps_n})^2+{\eps_n^2})\\
&=-\frac{1}{x_{n-k_n}^{\eps_n}} + (1-a) \log(x_{n-k_n}^{\eps_n})  + o(1) \\
&=w^o(x_{n-k_n}^{\eps_n},y_{n-k_n}^{\eps_n}) + o(1).
\end{align*}
using in the second line that $k_n^3/n^2=o(1)$	and in the last line the asymptotic expansion of $w^o$. Similarly, by the computation above,
\begin{align*}
t_{\eps_n}(x_{n-k_n}^{\eps_n},y_{n-k_n}^{\eps_n}) =\frac{y_{n-k_n}^{\eps_n}}{(x_{n-k_n}^{\eps_n})^\eta} \left(1+\frac{\eps_n^2}{(x_{n-k_n}^{\eps_n})^2}\right)^{-\frac{\eta}2} =t^o(x_{n-k_n}^{\eps_n}, y_{n-k_n}^{\eps_n}) + o(1)
\end{align*}
using in the last line the asymptotic expansion of $t^o$ and the fact that $t^o(x_{n-k_n}^{\eps_n}, y_{n-k_n}^{\eps_n})=\ocal(1)$ because $(x_{n-k_n}^{\eps_n}, y_{n-k_n}^{\eps_n})\in\rcal_n(\hat C)$.
\end{proof}

\begin{lem}[Compare to Lemma~\ref{lem:firstkn}]\label{lem:takeRbigo}
There exists $N_1>0$ such that for $n$ large enough $\Phi_0^o\circ g_{\eps_n}^{n-N_1}(M_C)\subset U_n^o(r, C)$ (so in particular $g_{\eps_n}^{n-N_1}(M_C)\subset P^o(r,\widetilde  C)$) and
$$\Phi^o\circ g_{\eps_n}^{n-N_1}\bigr|_{M_C}=A_{k_n-N_1,0}\circ\Phi^o\circ g_{\eps_n}^{n-k_n}\bigr|_{M_C}+o(1),$$
where the convergence of the term $o(1)$ is uniform on $M_C$. 
\end{lem}

\begin{proof}
From the asymptotic expansion of $(\Phi^o)^{-1}$ (see Lemma \ref{lem:asympphio-1}), Proposition~\ref{prop:outgofatou} and the definitions of $\Phi_0^o$ and $U_n^o(R,\tilde C)$, there exists a constant $R_0>0$ such that 
\begin{equation}\label{eq:inclusion}
(\Phi^o)^{-1}\{(X,Y) \in -\mathcal{H}_{R_0,2}\times \D(0,Ce^{1/2+\pi |\re q| }): |X|<5 k_n\}\subset(\Phi_0^o)^{-1}(U_n^o(R,\tilde C)), 
\end{equation}
where $\mathcal{H}_{R_0,2}:=\left\{X\in\H_{R_0}:|\im X|<2\re X\right\}.$ Up to increasing $R_0$ if necessary, we assume that $2R_0\ge \widetilde K+1$, where $\widetilde K:=\max \{   |w^\iota(x,y)|  :  (x,y) \in M_C  \}$.
We fix some integer $N_1>\widetilde K+\re\sigma+2+R_0$.

Let us first prove that there exists a constant $K_3>0$ such that for $n$ large enough and for all $(x_0,y_0)\in(\Phi_0^o)^{-1}(U_n^o(R,\tilde C))$ such that $(x_1^{\eps_n},y_1^{\eps_n}):=g_{\eps_n}(x_0,y_0)\in P^o((2(R-3/2))^{-1},\widetilde  C+1)$ we have
\begin{equation}\label{eq:K_3o}
\|\Phi^o (x_1^{\eps_n},y_1^{\eps_n})-\Phi^o(x_0,y_0)-(1,0)\|\le K_3\left(\frac1n+\frac1{|x_0|^2n^2}\right).
\end{equation}
The computation is analogous to the one in the proof of Lemma~\ref{lem:firstkn}. Fix $(x_0,y_0)$ such that $(X_0,Y_0):=\Phi_0^o(x_0,y_0)\in U_n^o(r,\tilde C)$ and $(x_1^{\eps_n},y_1^{\eps_n})\in P^o((2(R-3/2))^{-1},\widetilde  C+1)$. Set $\Phi_G^o:=\Phi^o\circ(\Phi_0^o)^{-1}: -\mathbb{H}_{R-5/2} \times \D(0,\widetilde  C+2) \to \C^2$, so 
$$\Phi_G^o(X_1^o,Y_1^o)=\Phi_G^o(X_0,Y_0)+(1,0),$$
where $(X_1^o,Y_1^o):=\Phi_0^o\circ g\circ(\Phi_0^o)^{-1}(X_0,Y_0)$. Then
$$\Phi^o(x_1^{\eps_n},y_1^{\eps_n})-\Phi^o(x,y)-(1,0)=\Phi_G^o(X_1^{\eps_n,o},Y_1^{\eps_n,o})-\Phi_G^o(X_1^o,Y_1^o)$$
where $(X_1^{\eps_n,o},Y_1^{\eps_n,o}):=\Phi_0^o(x_1^{\eps_n},y_1^{\eps_n})$.
By Proposition~\ref{prop:outgofatou} and condition $(R_4)$, there exists a holomorphic map $v:  -\mathbb{H}_{R-5/2} \times \D(0,\widetilde  C+2) \to \C^2$ such that 
$$\Phi_G^o(X,Y)= (X-(1-a)\log (-X), Y) + v(X,Y),$$
with $\|v(X,Y)\| \leq 1$ for all $(X,Y) \in -\H_{R-5/2} \times \D(0,\widetilde  C+2)$. By Cauchy estimates, we have  
$$\| \partial_X v(X,Y)\|\leq 1  \quad \text{and} \quad \| \partial_Y v(X,Y)\|\leq 1  $$
for all $(X,Y) \in -\H_{R-3/2}\times \D(0,\widetilde  C+1)$ and hence, by the mean value inequality, 
$$\|v(X_1^{\eps_n,o},Y_1^{\eps_n,o}) - v(X_1^o,Y_1^o)\| \leq |X_1^{\eps_n,o}-X_1^o|+|Y_1^{\eps_n,o}-Y_1^o|.$$
On the other hand, also by the mean value inequality, we have 
$$|\log (-X_1^{\eps_n,o})-\log (-X_1^o)| < |X_1^{\eps_n,o}- X_1^o|. $$
Therefore
\begin{align*}
\|\Phi_G^o(X_1^{\eps_n,o},Y_1^{\eps_n,o})-\Phi_G^o(X_1^o,Y_1^o)\| &\leq |X_1^{\eps_n,o}-X_1^o|+|1-a|| \log (-X_1^{\eps_n,o}) - \log (-X_1^o)|\\
&+|Y_1^{\eps_n,o}-Y_1^o| + \| v(X_1^{\eps_n,o},Y_1^{\eps_n,o}) - v(X_1^o,Y_1^o)\| \\
 &\leq (2+|1-a|)\left(|X_1^{\eps_n,o}-X_1^o|+|Y_1^{\eps_n,o}-Y_1^o|\right).
\end{align*} 
Then using Lemma~\ref{lem:estimateGepso} we obtain \eqref{eq:K_3o} for $n$ large enough, with $K_3=2 (2+|1-a|)C^o$.

Now take $(x_0,y_0)\in M_C$. We will prove by induction on $j$ that for all $n$ large enough and for all $n-k_n \leq j \leq n-N_1$ we have\\[-8pt]
\begin{enumerate}
\item $\Phi_0^o  (x_j^{\eps_n},y_j^{\eps_n})\in U_n^o(R,\tilde C)$\\[-12pt]
\item $\displaystyle\| \Phi^o  (x_j^{\eps_n},y_j^{\eps_n})  - \Phi^o (x_{n-k_n}^{\eps_n},y_{n-k_n}^{\eps_n}) - (j-(n-k_n),0) \| \leq K_3 \sum_{k=n-k_n}^{j-1} \left(\frac{1}{n} + \frac{1}{|x_{k}^{\eps_n}|^2 n^2}\right)$.  
\end{enumerate} 
For $j=n-k_n$, the first statement follows from Lemma~\ref{prop:afteregg}, and there is nothing to prove for the second one. Let $n-k_n \leq j <n-N_1$ such that (1) and (2) hold for all $n-k_n \leq k \leq j$ and denote $(X_j^{\eps_n,o},Y_j^{\eps_n,o}):=\Phi_0^o(x_j^{\eps_n},y_j^{\eps_n})$. 
By equation~\eqref{eq:diffX1epsXo} we have
$$|X_{j+1}^{\eps_n,o}-X_j^{\eps_n,o}-1| < \frac{1}{2} \quad \text{and}\quad
|Y_{j+1}^{\eps_n,o}-Y_j^{\eps_n,o}|<M \left(\frac{1}{n}+ \frac{1}{|X_j^{\eps_n,o}|^2} \right).$$
Therefore, by definition of $U_n^o(R,\tilde C)$, we have for $n$ large enough
$$
\re(X_{j+1}^{\eps_n,o}) < -R + \frac{3}{2}, \quad |Y_{j+1}^{\eps_n,o}|< \widetilde  C + 1.
$$
It follows that $(x_{j+1}^{\eps_n}, y_{j+1}^{\eps_n}) \in P^o((2(R-3/2))^{-1}, \widetilde  C +1)$ so using~\eqref{eq:K_3o} and the induction hypothesis we get
$$\|\Phi^o  (x_{j+1}^{\eps_n},y_{j+1}^{\eps_n})-\Phi^o (x_{n-k_n}^{\eps_n},y_{n-k_n}^{\eps_n}) - (j+1-(n-k_n),0) \| \leq K_3 \sum_{k=n-k_n}^{j} \left(\frac{1}{n} + \frac{1}{|x_k^{\eps_n}|^2 n^2}\right)$$
and condition (2) is proved.
	
By definition of $U_n^o(R,\tilde C)$, we have $|x_k^{\eps_n}|^{-1} < 10 k_n$ for all $n-k_n \leq k \leq j$, therefore
\begin{align*}
\sum_{k=n-k_n}^{j} \left(\frac{1}{n}+\frac{1}{|x_k^{\eps_n}|^2 n^2}\right) \leq \frac{k_n}{n}+ \frac{100 k_n^3}{n^2} = o(1).
\end{align*}
Therefore, by Lemma~\ref{prop:afteregg} and Proposition~\ref{prop:passing} we have for $n$ large enough
$$\Phi^o(x_{j+1}^{\eps_n},y_{j+1}^{\eps_n})=\Phi_{\eps_n}^\iota(x_{k_n}^{\eps_n},y_{k_n}^{\eps_n})+(\sigma+j+1-n-k_n,0)+o(1)$$
so by Lemmas~\ref{lem:apprfatou} and \ref{lem:firstkn}
$$\Phi^o(x_{j+1}^{\eps_n},y_{j+1}^{\eps_n}) = (w^\iota(x_0,y_0)+\sigma +j+1-n, e^{\pi q} t^\iota(x_0,y_0))+ o(1).$$
Let $(w^o(x,y), t^o(x,y)):=\Phi^o(x,y)$. We then have, for $n$ large enough,
\begin{align*}
\re(w^o(x_{j+1}^{\eps_n},y_{j+1}^{\eps_n})) &\le \re(w^\iota(x_0,y_0))+\re\sigma+j+2-n \leq \widetilde K + \re\sigma+j+2-n \\
&\leq \widetilde K+\re\sigma+2 - N_1 < -R_0
\end{align*} 
and 
\begin{align*}
|\im(w^o(x_{j+1}^{\eps_n},y_{j+1}^{\eps_n}))|\le |\im(w^\iota(x_0,y_0))|+1\le \widetilde K+1\le 2R_0\le 2\re(-w^o(x_{j+1}^{\eps_n},y_{j+1}^{\eps_n}))
\end{align*}
so $\Phi^o(x_{j+1}^{\eps_n},y_{j+1}^{\eps_n})\in -\mathcal{H}_{R_0,2}$. Similarly, for $n$ large enough,
$$|w^o(x_{j+1}^{\eps_n},y_{j+1}^{\eps_n})| \leq|w^\iota(x_0,y_0)|+|\sigma|+ n-j\leq \widetilde K+|\sigma|+ n-j<5 k_n.$$
Finally, using the fact that $|y_0(-x_0)^{-\eta}|\le C-1$ by definition of $M_C$ and $|t^\iota(x_0,y_0)-y_0(-x_0)^{-\eta}|<1$ by condition $(R_4)$, we have that $|t^\iota(x_0,y_0)|\le C$ and then for $n$ large enough
$$|t^o(x_{j+1}^{\eps_n},y_{j+1}^{\eps_n})|\leq Ce^{1/2+\pi|\re q|}.$$
Therefore, by \eqref{eq:inclusion}, $(x_{j+1}^{\eps_n}, y_{j+1}^{\eps_n}) \in U_n^o(R,\tilde C)$ and (1) is proved. Taking $j=n-N_1$, the Lemma~\ref{lem:takeRbigo} is proved.
\end{proof}

\subsection{Conclusion}

\begin{proof}[Proof of Theorem~\ref{th:explicit}]
	We have 
	\begin{align*}
		\Phi^o\circ g_{\eps_n}^{n-N_1}\bigr|_{M_C} &= A_{k_n-N_1,0}\circ \Phi^o\circ g_{\eps_n}^{n-n_k}\bigr|_{M_C}+ o(1) 
		\text{ \quad by Lemma~\ref{lem:takeRbigo}}   \\
		&=A_{k_n-N_1,0}\circ \Phi_{\eps_n}^o\circ g_{\eps_n}^{n-k_n}\bigr|_{M_C}+o(1)
		\text{ \quad by Lemma~\ref{prop:afteregg} } \\
		&=A_{k_n-N_1,0}\circ A_{\sigma-2k_n,q}\circ \Phi_{\eps_n}^\iota\circ g_{\eps_n}^{k_n}\bigr|_{M_C}+o(1)   
		\text{ \quad by Proposition~\ref{prop:passing} }\\
		&=A_{k_n-N_1,0}\circ A_{\sigma-2k_n,q}\circ\Phi^\iota \circ g_{\eps_n}^{k_n}\bigr|_{M_C}+o(1) \text{\quad by Lemma \ref{lem:apprfatou}}\\
		&=A_{k_n-N_1,0}\circ A_{\sigma-2k_n,q}\circ A_{k_n,0}\circ \Phi^\iota\bigr|_{M_C}+o(1) \text{\quad by Lemma \ref{lem:firstkn}}\\
		&=A_{\sigma-N_1,q}\circ \Phi^\iota\bigr|_{M_C}+o(1)
	\end{align*}
	where the convergence of the terms $o(1)$ is uniform in $M_C$. Therefore
	\begin{equation*}\label{eq:cv}
		g_{\eps_n}^{n-N_1}\bigr|_{M_C}= (\Phi^o)^{-1} \circ A_{\sigma-N_1, q} \circ \Phi^\iota\bigr|_{M_C} + o(1)
	\end{equation*}
	where the convergence of the term $o(1)$ is uniform on $M_C$. This proves that $g_{\eps_n}^{n-N_1} \to \lcal_{\sigma-N_1,q}$ uniformly on $M_C$, where $\lcal_{\sigma-N_1, q}:=(\Phi^o)^{-1} \circ {A}_{\sigma-N_1, q} \circ \Phi^\iota$.

	Let us now explain how we can deduce the same convergence statement first for all $(x,y) \in L$, then for all $(x,y) \in K$.
	First, assume without loss of generality that the domain $U$ on which the maps $g_{\eps_n}$ are defined 
	is a bi-disk $\D(0,\delta)^2$. 
	
Let $\pi_i: \C^2 \to \C$ ($1 \leq i \leq 2$) denote the projections on each coordinate of $\C^2$. 
Let $x \in \pi_1(L)$. Recall that     $r$ and $C$ are the constants introduced in Definition \ref{defi:choiceofR}. Let
$$M_x(C):=\{y \in \C: (x,y) \in M_C  \} =  \overline{\D}\left(0, (C-1) |(-x)^\eta| \right)  \setminus \D(0, (C-1)^{-1} |(-x)^\eta|)   $$
and 
$$N_x(C):=\{y \in \C: (x,y) \in N_C  \}  =  \overline{\D}\left(0, (C-1) |(-x)^\eta| \right).   $$
	
	Let us prove inductively on $0 \leq j \leq n-N_1$ that $g_{\eps_n}^j(N_C) \subset U$.
	For $j=0$, there is nothing to prove, since $N_C \subset P^\iota(r,C) \subset U$.
	Let $0 \leq j \leq n-N_1-1$ be such that  $g_{\eps_n}^j(N_C) \subset U$, so that $g_{\eps_n}^{j+1}$ is well-defined on $N_C$.
	
	By the maximum principle, for all $x \in \pi_1(L) = \pi_1(N_C)$, we have 
	$$\sup_{y \in N_x(C)}  |\pi_1 \circ g_{\eps_n}^{j+1}(x,y)| \leq  \max_{y \in \partial N_x(C) }   |\pi_1 \circ g_{\eps_n}^{j+1}(x,y)| \leq \max_{y \in M_x(C) }   |\pi_1 \circ g_{\eps_n}^{j+1}(x,y)| \leq   \delta.$$
	Similarly, we also have 
	$$\sup_{y \in N_x(C)}  |\pi_2 \circ g_{\eps_n}^{j+1}(x,y)| \leq  \max_{y \in \partial N_x(C) }   |\pi_2 \circ g_{\eps_n}^{j+1}(x,y)| \leq \max_{y \in M_x(C) }   |\pi_2 \circ g_{\eps_n}^{j+1}(x,y)| \leq   \delta.$$
	
Therefore, $g_{\eps_n}^j$ is well-defined on $N_C$ for every $0 \leq j \leq n-N_1$.
Moreover, $(g_{\eps_n}^{n-N_1}: N_C \to \C^2)_{n \geq 0}$ is a normal family in the sense of Montel. 
Let $(n_k)_{k \geq 0 }$ be any extracted sequence such that $g_{\eps_{{n_k}}}^{{n_k}-N_1}$ converges on $N_C$ to some holomorphic function $G$. We have proved that $G=\lcal_{\sigma-N_1,q}$ on $M_C$, therefore, by the identity principle, $G=\lcal_{\sigma-N_1,q}$ on all of $N_C$. Since this is true for any converging subsequence, we conclude that $g_{\eps_n}^{n{-N_1}} \to \lcal_{\sigma-N_1,q}$ on
all of $N_C$, hence on $L$ (since $L \subset N_C$).
	
Let us now prove that $g_{\eps_n}^{n-N} \to \lcal_{\sigma-N,q}$ on $K$, where $N:=N_1-n_0$. Indeed, by the definition of $L$, we have that for all $n$ large enough $g_{\eps_n}^{n_0}(K) \subset L$; therefore, if $(x,y) \in K$ then 
$$\limn g_{\eps_n}^{n-N_1+n_0}(x,y)= \limn g_{\eps_n}^{n-N_1} \circ g_{\eps_n}^{n_0}(x,y) = \lcal_{\sigma-N_1,q} \circ {g}^{n_0}(x,y) = \lcal_{\sigma-N_1+n_0,q}(x,y),$$
using in the last two equalities the fact that $g_{\eps_n}^{n_0}$ converges uniformly to ${g}^{n_0}$ on $K$  and that 
$\lcal_{\sigma+1,q} = g \circ \lcal_{\sigma,q} = \lcal_{\sigma, q} \circ g$. 
\end{proof}

\section{Proof of Corollaries \ref{coro:discbigjulia} and  \ref{coro:endo}}\label{sec:coro}

The goal of this section is to prove Corollaries \ref{coro:discbigjulia} and  \ref{coro:endo}.
We start with a general observation about Fatou coordinates for globally defined maps.

\begin{prop}\label{prop:discrete}
Assume that the germ $g:=g_0$ from Theorem~\ref{th:explicit} extends to a holomorphic self-map of a complex manifold.  Then the incoming Fatou coordinate $\Phi^\iota$ extends to a holomorphic map on the parabolic basin $\bcal_{v}$ associated to $v=(1,0)$, and the map
$\left(\Phi^o \right)^{-1}$ extends to a holomorphic map $\Psi^o$ on $\C^2$. Moreover, if $f$ has discrete fibers, then so do 
$\Phi^\iota$ and $\Psi^o$. 
\end{prop}

\begin{proof}
Recall that $\bcal_{v}=\bigcup_{C>0}\bigcup_{n \geq 0} g^{-n}(P^\iota(r,C))$ for any $0<r\le r^\iota(C)$. Then, using the functional equation $\Phi^\iota \circ {g}= \Phi^\iota + (1,0)$ we can extend 
$\Phi^\iota$ to $\bcal$ by 	
$$\Phi^\iota(z) := \Phi^\iota \circ g^n(z) - (n,0)$$ 
for any $n$ such that $g^n(z)\in P^\iota(r,C)$. Moreover, by Proposition~\ref{prop:outgofatou} we have that for any $C>0$ there exists $r=r^o(C)>0$ such that 
$$\{ X \in -\H_{r^{-1}} :  |\im(X)| < 2 |\re(X)|   \} \times \D(0,C-1)  \subset \Phi^o(P^o(r,C)) $$
so we can extend $(\Phi^o)^{-1}$ to a map $\Psi^o:\C\times\D(0,C-1)\to\C$ by
$$\Psi^o(X,Y):=g^n \circ (\Phi^o)^{-1}(X-n,Y)$$
for any $n\in\N$ such that $(X-n,Y)\in -\H_{r^{-1}}\times \D(0,C-1)$. Since we can do this for any $C>0$, the map $\Psi^o$ extends to a holomorphic map on $\C^2$.
	
Let us now prove the claim about the discreteness of fibers. Take $p \in \bcal$ and let $F_p:=(\Phi^\iota)^{-1}(\{\Phi^\iota(p)\})$ be the corresponding fiber. Let $n$ be large enough such that $g^n(p) \in P^\iota(r,C)$ for some $C>0$ and $r\le r^\iota(C)$. Then 
	$$\Phi^\iota \circ g^n(F_p) = \Phi^\iota(F_p) + (n,0) = \{\Phi^\iota(p)\} + (n,0).$$
	Since $\Phi^\iota$ is injective on $P^\iota(r,C)$, the point $f^n(p)$ is isolated in $f^n(F_p)$. Since $g$ (hence $g^n$) has discrete fibers, $p$ must also be isolated in $F_p$. 
	
	Similarly, if $(X_0,Y_0) \in \C^2$ then there exists $n \in \N$ such that $(X_0-n,Y_0) \in \Phi^o(P^o(r,C))$ for some $C>0$ and $r\le r^o(C)$. In order to prove that $(X_0,Y_0)$ is isolated in its fiber $F_{(X_0,Y_0)}:=(\Psi^o)^{-1}(\{\Psi^o(X_0,Y_0)\})$, it suffices to prove that $(X_0-n,Y_0)$ is isolated in $F_{(X_0,Y_0)}-(n,0)$. By the relation $g^n \circ \Psi^o(X-n,Y) = \Psi^o(X,Y)$, we have that $F_{(X_0,Y_0)}-(n,0) = (\Psi^o)^{-1}(E_n)$, where $E_n:=g^{-n}( \{ \Psi^o(X_0,Y_0)\} )$. Since $g$ (hence $g^n$) has discrete fibers, $E_n$ is discrete. Since $(X_0-n,Y_0) \in \Phi^o(P^o(r,C))$ and $\Phi^o$ is injective, $(X_0-n,Y)$ is therefore isolated in $(\Psi^o)^{-1}(E_n)=F_{(X_0,Y_0)}- (n,0)$. Thus $F_{(X_0,Y_0)}$ is indeed discrete. 
\end{proof}

Next, we prove the following Lemma: 

\begin{lem}\label{lem:existenceperturbq}
	Let $f$ be an endomorphism of $\ptwo$ of algebraic degree $d$, satisfying $(H_1)$ and such that $d > \re \alpha+1$.
	Then for any $q \in \C$ there exists a family $(f_\eps)_{\eps \in \D}$ of endomorphisms of degree $d$ satisfying $(H_1)-(H_3)$, with $f_0=f$ and $q$ as in Theorem \ref{th:v2}.
\end{lem}

\begin{proof}		
	We work in some affine coordinates in which the fixed point of $f$ is the origin and in which the non-degenerate characteristic direction is $[1:0]$ so the non-singular formal invariant curve $\mathcal{C}$ admits a parametrization of the form $\gamma(t)=(t, \zeta(t))$. As in Section~\ref{sec:diffvers}, we let $\Psi(x,y):=(x, y-J_{m+1}\zeta(x))$, where $J_{m+1}\zeta$ is the jet of order $m+1$ of $\zeta$, where $m:=\lfloor \re\alpha \rfloor+1$ of $\zeta$, and we let $g:=\Psi \circ f \circ \Psi^{-1}$. Then $g$ is of the form 
	$$g(x,y) = ( x+ x^2 a(x) + y b(x,y), y (1+ c(x,y) ) + \ocal(x^{m+2}) ),$$
	where $a(0)=1$ and $b(0,0) = c(0,0)=0$. We now let 
	$$g_\eps(x,y):= ( x+ (x^2+\eps^2) a(x) + y b(x,y), y (1+ c(x,y) + q \eps ) + \ocal(x^{m+2}) ),$$
	and $\widetilde  f_\eps := \Phi^{-1} \circ g_\eps \circ \Phi$. By construction, $(\widetilde  f_\eps)_{\eps \in \D}$ 
	satisfies $(H_1)-(H_3)$ and $f_0 = f$. However, the maps $\widetilde  f_\eps$ need not be endomorphisms of $\ptwo$, 
	since $\Psi$ is not an automorphism of $\ptwo$; we will therefore need to make one last modification.
	Let us write $\widetilde  f_\eps(x,y)= f(x,y) + \eps h(\eps,x,y)$, where $h(\eps,x,y) = \sum_{i,j,k \in \N} a_{i,j,k} \eps^i x^j y^k$ is a holomorphic map defined on some neighborhood of $0$ in $\C^3$. Let $h_{m+1}(\eps,x,y):= \sum_{i,j,k \leq m+1}  a_{i,j,k} \eps^i x^j y^k$ be the jet of order $m+1$ of $h$ and $f_\eps(x,y):= f(x,y) + \eps h_{m+1}(\eps,x,y)$. It is not difficult to see that conditions $(H_1)-(H_3)$ only depend on the jet of order $m+1$ of $f_\eps$ at $(0,0,0)$; therefore, the family of maps $(f_\eps)_{\eps \in \D}$ still satisfies $(H_1)-(H_3)$, and since $m<d$ by assumption, the maps $f_\eps$ are endomorphisms of $\ptwo$ for all $\eps$ in a neighborhood of $0$. 
	
Finally, for the fixed points $z_1(\eps)=1+2i\eps+\ocal(\eps^2)$ and $z_2(\eps)=1-2i\eps+\ocal(\eps^2)$ we have that
$$\jac\, f_\eps(z_1(\eps))=\begin{pmatrix}
1+2i\eps+\ocal(\eps^2)  &    \ocal(\eps) \\
\ocal(\eps^2) & 1+q \eps + i \eta \eps + \ocal(\eps^2)
\end{pmatrix}$$	
and 
$$\jac\, f_\eps(z_2(\eps))=\begin{pmatrix}
1-2i\eps+\ocal(\eps^2)  &    \ocal(\eps) \\
\ocal(\eps^2) & 1+q \eps - i \eta \eps + \ocal(\eps^2)
\end{pmatrix}.$$
so clearly $q$ is as in Theorem \ref{th:v2} by Lemma~\ref{lem:pertid}.
\end{proof}

We now complete the proof of Corollary~\ref{coro:discbigjulia}, which is essentially the same done by Bianchi in \cite{bianchi19parabolic} or the one due to Lavaurs (\cite{lavaurs1989systemes}) in dimension 1:

\begin{proof}[Proof of Corollary \ref{coro:discbigjulia}]
	If we take $q\in\C$, the existence of a family $(f_\eps)_{\eps\in\D}$ of endomorphisms of degree $d$ satisfying $(H_1)-(H_3)$ and such that $q:=\lim_{\eps\to0}\frac1\eps\,\frac{\rho_N^1(\eps)+\rho_N^2(\eps)-2}{\rho_T^1(\eps)+\rho_T^2(\eps)}$ is guaranteed by Lemma~\ref{lem:existenceperturbq}. Now, take $\sigma\in\C$ and $z_0 \in J_1(f, \lcal_{\sigma,q})$. Without loss of generality, assume that $z_0 = \lcal_{\sigma, q}^N(p_0)$, for some $N \in \N$ and some $p_0 \in J_1(f)$. Consider a sequence $(\eps_n)$ as in Theorem~\ref{th:v2}. We will find a sequence $z_n \in J_1(f_{\eps_n})$ such that $\limn z_n = z_0$. 
	
	By lower semicontinuity of $\eps \mapsto J_1(f_\eps)$, there exists a sequence $p_n \in J_1(f_{\eps_n})$ such that $\limn p_n = p_0$.
	Let $z_n:=f_{\eps_n}^{n N}(p_n) \in J_1(f_{\eps_n})$: by Theorem~\ref{th:v2}, we have $\limn z_n = \lcal_{\sigma, q}^N(p_0) = z_0$, and we are done.
\end{proof}

We can now prove Corollary \ref{coro:endo}. The idea is that thanks to the Main Theorem, we can find a suitable sequence of perturbations $f_{\eps_n}$ such that $f_{\eps_n}^n$ maps a point in $\bcal$ to a repelling periodic point in $J_2(f_{\eps_n})$. Since $J_2(f_{\eps_n})$ is completely invariant, this will mean that 
$\bcal$ contains some points in $J_2(f_{\eps_n})$ for all $n$ large enough, thus proving the lack of upper semicontinuity of $\eps \mapsto J_2(f_\eps)$. There is however a technical issue: contrary to the case of rational maps on $\rs$,
not all repelling periodic points belong to $J_2(f_\eps)$. 

\begin{proof}[Proof of Corollary \ref{coro:endo}]
	Let $r(0)$ denote a repelling periodic point for $f$ contained in $\Psi^o(\C^2)$. Let $(X_1,Y_1) \in \C^2$ be such that $\Psi^o(X_1,Y_1) = r(0)$.	
	Let $z_0 \in \bcal$ be such that $\Phi^\iota(z_0) \in \C \times \{0\}$ if $Y_1=0$, and $\Phi^\iota(z_0) \in \C \times \C^*$ {if $Y_1\neq 0$}. Then there exists a unique $(\sigma,q) \in \C^2$ such that $A_{\sigma,q} \circ \Phi^\iota(z_0) = (X_1, Y_1)$, or	in other words,
	$\lcal_{\sigma,q}(z_0)  = r(0)$. Let $(f_\eps)_{\eps \in \D}$ be the corresponding family of perturbations of $f$ constructed in Lemma \ref{lem:existenceperturbq} for that value of $q$.
	Let $r(\eps)$ be the repelling periodic point of $f_{\eps}$, which moves holomorphically with $\eps$ for $|\eps|$ small. For $|\eps|$ small,  we have $r(\eps) \in J_2(f_{\eps})$ by \cite[Lemma 4.9]{bianchi2019misiurewicz}. 
	
Let $\sigma_0 \in \C$ be given by Theorem \ref{th:v2}, and let 
$\eps_n \to 0$ be such that 
$$\frac{2i\pi}{\rho_T^1(\eps_n)-1} = n - (\sigma-\sigma_0) + o(1).$$
Let $U$ be a neighborhood of $z_0$ in $\bcal$, small enough such that $\lcal_{\sigma,q}(z)=r(0)$ and $z \in U$ implies that $z=z_0$. Then the maps $G_n(z)=f_{\eps_n}^n(z)-r(\eps_n)$ converge locally uniformly to $G(z)=\lcal_{\sigma,q}(z)-r(0)$ by Theorem~\ref{th:v2}. Since $(0,0)$ is an isolated point in $G^{-1}(\{(0,0)\})$ by Proposition \ref{prop:discrete}, there is a sequence $z_n \to (0,0)$ such that $G_n(z_n) = (0,0)$. By the backward invariance of the small Julia set, $z_n \in J_2(f_{\eps_n})$,  but $U \cap J_2(f)=\emptyset$ since $U \subset \bcal$.
\end{proof}

\bibliographystyle{alpha} 
\bibliography{biblio}

\end{document}